\documentclass[a4paper]{amsart}

\usepackage[mode=image|tex]{standalone}
\usepackage[left]{showlabels}
\usepackage[T1]{fontenc}	%with this can use \k in bibliography

\usepackage[T1]{fontenc}	
\usepackage[utf8]{inputenc}			% can use \k{}
\usepackage{amsmath,amssymb,amsthm,amsbsy}
\usepackage{mathtools} 					%typsetting tools
\usepackage{microtype} 					%better looking pdf output
\usepackage{bbm}					%more blackboard symbols
\usepackage{bm}						%more boldface symbols
\usepackage[shortcuts]{extdash} 			% hyphen \=/
\usepackage{color}
\usepackage{enumerate}
\usepackage{graphicx}    
\usepackage{xspace}					%for defining macros with or without the blank space at the end
\usepackage{subcaption}					%necessary for subfigures
\usepackage[colorlinks=true,linkcolor=black, citecolor=black]{hyperref}  
\usepackage{mparhack}					%fix a bug relating \marginpar 
\usepackage{accents}
\usepackage[mode=image|tex]{standalone}

\usepackage{tikz}
  \usetikzlibrary{arrows,matrix,decorations.markings,patterns,cd}
  \tikzset{
    >=latex,
    bnode/.style={circle,fill=black,draw=black},
    middlearrow/.style={	%arrow in the middle of the line
	  decoration={markings, mark= at position 0.55 with {\arrow{#1}}},
	  postaction={decorate}
	  },
    double arrow/.style args={#1 colored by #2 and #3}{ %code for double arrows taken online
	  -stealth,line width=#1,#2, % first arrow
	  postaction={draw,-stealth,#3,line width=(#1)/3,
                shorten <=(#1)/3,shorten >=2*(#1)/3}, % second arrow
	  }
	  
  }
\tikzset{
        hatch distance/.store in=\hatchdistance,
        hatch distance=10pt,
        hatch thickness/.store in=\hatchthickness,
        hatch thickness=2pt
    }
    \makeatletter
    \pgfdeclarepatternformonly[\hatchdistance,\hatchthickness]{flexible hatch}
    {\pgfqpoint{0pt}{0pt}}
    {\pgfqpoint{\hatchdistance}{\hatchdistance}}
    {\pgfpoint{\hatchdistance-1pt}{\hatchdistance-1pt}}%
    {
        \pgfsetcolor{\tikz@pattern@color}
        \pgfsetlinewidth{\hatchthickness}
        \pgfpathmoveto{\pgfqpoint{0pt}{0pt}}
        \pgfpathlineto{\pgfqpoint{\hatchdistance}{\hatchdistance}}
        \pgfusepath{stroke}
    }
    
    \pgfdeclarepatternformonly[\hatchdistance,\hatchthickness]{flexible revhatch}
    {\pgfqpoint{0pt}{0pt}}
    {\pgfqpoint{\hatchdistance}{-\hatchdistance}}
    {\pgfpoint{\hatchdistance-1pt}{\hatchdistance-1pt}}%
    {
        \pgfsetcolor{\tikz@pattern@color}
        \pgfsetlinewidth{\hatchthickness}
        \pgfpathmoveto{\pgfqpoint{0pt}{0pt}}
        \pgfpathlineto{\pgfqpoint{\hatchdistance}{-\hatchdistance}}
        \pgfusepath{stroke}
    }

		\newcommand{\RR}{\mathbb{R}}	
\renewcommand{\SS}{\mathbb{S}}		\newcommand{\TT}{\mathbb{T}}

		\newcommand{\ZZ}{\mathbb{Z}}	

\newcommand{\CA}{\mathcal{A}}			
\newcommand{\CC}{\mathcal{C}}			
			
\newcommand{\CG}{\mathcal{G}}

\newcommand{\BA}{\mathbf{A}}		\newcommand{\BB}{\mathbf{B}}

\newcommand{\BK}{\mathbf{K}}

		\newcommand{\BX}{\mathbf{X}}

\newcommand{\fka}{\mathfrak{a}}		\newcommand{\fkb}{\mathfrak{b}}

\newcommand{\note}[1]{ \marginpar{\footnotesize #1}}    %side note   (add {\iffalse... \fi} to hide all notes)

\newcommand{\ubar}[1]{\underaccent{\bar}{#1}} 		%\bar under letter

\makeatletter						%defines left and right double angle brackets
\DeclareFontFamily{OMX}{MnSymbolE}{}	
\DeclareSymbolFont{MnLargeSymbols}{OMX}{MnSymbolE}{m}{n}
\SetSymbolFont{MnLargeSymbols}{bold}{OMX}{MnSymbolE}{b}{n}
\DeclareFontShape{OMX}{MnSymbolE}{m}{n}{
    <-6>  MnSymbolE5    <6-7>  MnSymbolE6   <7-8>  MnSymbolE7
   <8-9>  MnSymbolE8   <9-10> MnSymbolE9  <10-12> MnSymbolE10
  <12->   MnSymbolE12	}{}
\DeclareFontShape{OMX}{MnSymbolE}{b}{n}{
    <-6>  MnSymbolE-Bold5   <6-7>  MnSymbolE-Bold6   <7-8>  MnSymbolE-Bold7
   <8-9>  MnSymbolE-Bold8   <9-10> MnSymbolE-Bold9  <10-12> MnSymbolE-Bold10
  <12->   MnSymbolE-Bold12  }{}
\let\llangle\@undefined
\let\rrangle\@undefined
\DeclareMathDelimiter{\llangle}{\mathopen}%
                     {MnLargeSymbols}{'164}{MnLargeSymbols}{'164}
\DeclareMathDelimiter{\rrangle}{\mathclose}%
                     {MnLargeSymbols}{'171}{MnLargeSymbols}{'171}
\makeatother

\DeclarePairedDelimiter\abs{\lvert}{\rvert}		%absolute value
\DeclarePairedDelimiter\norm{\lVert}{\rVert}		%norm
\DeclarePairedDelimiter\angles{\langle}{\rangle}	% used in \scal
\DeclarePairedDelimiter\aangles{\llangle}{\rrangle}	% normal closure 			!!   CHECK  !! 
\DeclarePairedDelimiter\paren{(}{)}			%   These are to fix the spacing between 
\DeclarePairedDelimiter\braces{\{}{\}}			%   Also, use with * to get variable size

\makeatletter
\let\oldabs\abs	
\def\abs{\@ifstar{\oldabs}{\oldabs*}}			%makes it default to rescale the size of \abs (resp \norm). Use \abs* (resp \norm*) if you don't want them to rescale
\let\oldnorm\norm
\def\norm{\@ifstar{\oldnorm}{\oldnorm*}}
\makeatother

	\newcommand{\bigparen}[1]{\paren[\big]{#1}}
	\newcommand{\Bigparen}[1]{\paren[\Big]{#1}}

\newcommand{\bigbrace}[1]{\braces[\big]{#1}}	
	
\newcommand{\bigmid}{\mathrel{\big|}}			%bigger \mid
\newcommand{\thhom}{\mathrel{\sim_\theta}}

\makeatletter
\def\thgrp{\@ifnextchar[{\@thgrpwith}{\@thgrpwithout}} %theta-fundamental group
\def\@thgrpwith[#1]{\pi_{1,#1}}
\def\@thgrpwithout{\pi_{1,\theta}}
\def\thhom{\@ifnextchar[{\@thhomwith}{\@thhomwithout}} %theta-homotopy
\def\@thhomwith[#1]{\mathrel{\sim_{#1}}}
\def\@thhomwithout{\mathrel{\sim_\theta}}
\makeatother

\newsavebox{\discrmapgraphic}  %creates a box containing widehat and a dot of the right size
\savebox{\discrmapgraphic}{%  <---- note that this % is important: otherwise an empty space is added to the box
  \begin{tikzpicture}[scale=1, every node/.style={transform shape}]
      \path[clip] (-3.1 pt, -2.8 pt) rectangle (3.1 pt, 4 pt);
      \fill (0,0) circle (0.78 pt) (0,-0.6 pt) node[draw=black] {$\widehat{}$};
      \end{tikzpicture}% <---- and this % is equally important.
  }

\DeclareMathOperator{\lk}{{\rm Lk}} 			%for links
\DeclareMathOperator{\vertex}{V}			%vertex set
\DeclareMathOperator{\SO}{SO}

\DeclareMathOperator*{\Ast}{\scalebox{1.5}{\raisebox{-0.2ex}{$\ast$}}} %big \ast for multiple joins

\DeclareMathAlphabet{\mathbit}{OT1}{cmr}{bx}{it}  	% bold italic in math mode  (note that \boldsymbol{} makes EVERYTHING bold in math mode)

\theoremstyle{plain}
\newtheorem{thm}{Theorem}[section]		\newtheorem{theorem}[thm]{Theorem}		
\newtheorem{prop}[thm]{Proposition}		
\newtheorem{lem}[thm]{Lemma}						
\newtheorem{cor}[thm]{Corollary}		
\newtheorem{qu}[thm]{Question}			\newtheorem{question}[thm]{Question}
		
\newtheorem{fact}[thm]{Fact}

\newtheorem*{thm*}{Theorem}			\newtheorem*{theorem*}{Theorem}		
\newtheorem*{prop*}{Proposition}		\newtheorem*{proposition*}{Proposition}
\newtheorem*{lem*}{Lemma}			\newtheorem*{lemma*}{Lemma}			
\newtheorem*{cor*}{Corollary}			\newtheorem*{corollary*}{Corollary}
\newtheorem*{qu*}{Question}			\newtheorem*{question*}{Question}
\newtheorem*{conj*}{Conjecture}			\newtheorem*{conjecture*}{Question}
\newtheorem*{fact*}{Fact}
\newtheorem*{claim*}{Claim}

\newtheorem{alphthm}{Theorem}			%letter numbering

\newtheorem{alphprop}[alphthm]{Proposition}

\theoremstyle{definition}
\newtheorem{de}[thm]{Definition}			
	
		\newtheorem{convention}[thm]{Convention}

\newtheorem*{de*}{Definition}			\newtheorem{definition*}{Definition}	
\newtheorem*{notation*}{Notation}	
\newtheorem*{conv*}{Convention}			\newtheorem*{convention*}{Convention}

\theoremstyle{remark}
\newtheorem{rmk}[thm]{Remark}			\newtheorem{remark}[thm]{Remark}			
\newtheorem{exmp}[thm]{Example}

\usepackage{cleveref}

\theoremstyle{remark}

\usepackage{subcaption}
\usepackage{color}

\newcommand{\hpl}{\hat{\mathfrak{h}}}
\newcommand{\hplv}{\hat{\mathfrak{v}}}
\newcommand{\hplt}{\hat{\mathfrak{g}}}

\newcommand{\hsp}{\mathfrak{h}}

\newcommand{\xcplx}{{\BX_{\Gamma_A,\Gamma_B}}} %X-complex
\newcommand{\xLcplx}{{\BX_{\Lambda_A,\Lambda_B}}}
\newcommand{\tildexcplx}{{\widetilde{\BX}_{\Gamma_A,\Gamma_B}}} %X-complex
\newcommand{\simp}{{\rm Simp}} 
\newcommand{\flag}{{\rm Flag}} 
\newcommand{\npc}{{\rm NPCCC}} 
\newcommand{\cc}{{\rm CC}} 
\newcommand{\loc}{{\rm loc}} 

\newcommand{\oldnote}[1]{\note{#1}}
\renewcommand{\oldnote}[1]{\iffalse... \fi} 

\title{A new construction of CAT(0) cube complexes}

\author{Robert Kropholler and Federico Vigolo}             
\thanks{F.V. was also funded by the EPSRC Grant 1502483; the J.T. Hamilton Scholarship; the Isaac Newton Institute for Mathematical Sciences, Cambridge (EPSRC grant
no EP/K032208/1), for support and hospitality during the programme “Non-positive curvature, group
actions and cohomology”; and the ISF Moked 713510 grant
number 2919/19.}

\begin{document}

	\begin{abstract}
		We introduce the notion of coupled link cube complex (CLCC) as a means of constructing interesting cocompactly cubulated groups. CLCCs are defined locally, making them a useful tool when precise control over the links is required. In this paper we study some general properties of CLCCs, such as their (co)homological dimension and criteria for hyperbolicity. Some examples of fundamental groups of CLCCs are RAAGs, RACGs, surface groups and some manifold groups. 
		As immediate applications of our criteria we produce a number of cubulated $3$ and $4$ manifolds with hyperbolic fundamental group.
	\end{abstract}
	
	\maketitle
	
	\tableofcontents
	
	\section{Introduction}
	\label{sec:introduction}
	
	The class of finitely presented groups is wild and full of mysterious objects\footnote{The large ``hic abundant leones'' region in Bridson's map of all groups.}, and studying it is one of the central tenets of geometric group theory. Despite its wildness, it is generally complicated to produce concrete examples of groups satisfying pathological properties because the groups we can actually describe and understand are, of course, relatively well-behaved. One of the consequences of this scarcity of examples is that a number of fundamental questions remain wide open and out of reach.
	
	A general strategy for constructing interesting examples is to restrict one's attention to a special class of groups where it is possible to employ some specialized tools. One such family is that of groups acting geometrically on CAT(0) cube complexes. There, it is possible to use both the CAT(0) geometry and the combinatorial data given by the hyperplanes in order to gain substantial knowledge and constraints on the algebraic structure of the acting group. Yet, it is still complicated to explicitly construct many pathological examples.
	
	Two hugely successful subclasses in the class of groups acting on CAT(0) cube complexes are right angled Artin groups (RAAGs) and right angled Coxeter groups (RACGs). Among other things, they have been used to construct groups with various finiteness properties \cite{bestvina_morse_1997}, exotic aspherical manifolds \cite{davis1983groups} and non\=/arithmetic hyperbolic groups of high cohomological dimension \cite{januszkiewicz_hyperbolic_2003,osajda_construction_2013}.\oldnote{More examples?}

	Part of the reason why the study of RAAGs and RACGs is so fruitful is that they are in natural correspondence with finite (flag) simplicial complexes. Namely, every simplicial complex $\Gamma$ uniquely determines a RAAG $\CA_\Gamma$ and a RACG $\CC_\Gamma$ and, vice versa, every RAAG and RACG is obtained this way. Moreover  the groups $\CA_\Gamma$ and $\CC_\Gamma$ act geometrically on nice CAT(0) cube complexes: the Salvetti and Davis complexes respectively. One particularly good feature of these cube complexes is that all the vertices have isomorphic links. In fact, every vertex in the Davis complex of $\CC_\Gamma$ has link isomorphic to $\Gamma$ itself, while every vertex in the Salvetti complex of $\CA_\Gamma$ has link isomorphic to the `octahedralisation' of $\Gamma$. This precise geometric understanding of the classifying spaces of RAAGs and RACGs can then be used to translate desirable geometric/algebraic properties into combinatorial properties of the defining complex $\Gamma$. Using this dictionary, it is possible to produce concrete examples of groups with prescribed properties by finding appropriate simplicial complexes.
	
	On the other hand, such precise geometric control comes at the cost of extra rigidity. This makes the classes of RAAGs and RACGs ill\=/suited to produce some examples of particularly pathological groups. The main objective of this work is to introduce a new and more flexible class of groups acting on CAT(0) cube complexes. This class contains both RAAGs and finite index subgroups of RACGs, and consists of groups that are defined very explicitly via an easy\=/to\=/describe classifying space.
	
	\
	
	In this paper we consider a construction that takes as input a pair of finite $n$\=/coloured flag simplicial complexes $\Gamma_A$ and $\Gamma_B$, and it outputs a finite non\=/positively curved cube complex $\xcplx$ (see Section~\ref{sec:complex_X_Gamma} for the complete definition). 
	We call $\xcplx$ a \emph{coupled link cube complex}, which shortens to \emph{CLCC}. When $\xcplx$ is connected, its universal cover $\tildexcplx$ is a CAT(0) cube complex, hence contractible. In particular, $\xcplx$ is a classifying space for its fundamental group. We can thus use topological and geometric properties of $\xcplx$ and $\tildexcplx$ to investigate the algebraic properties of $\pi_1(\xcplx)$.

	The most important feature of this construction is that it allows us to have complete control over links of the vertices in the cube complex. That is, the link at a vertex $v\in\xcplx$ can be expressed as a simplicial join 
	\begin{equation}
	\label{eq:intro}
	\lk\bigparen{v,\xcplx}=\lk\bigparen{\sigma_A(v),\Gamma_A}\ast\lk\bigparen{\sigma_B(v),\Gamma_B}  \tag{$\bigstar$}
	\end{equation}
	where $\sigma_A(v)$ and $\sigma_B(v)$ are two appropriate simplices in $\Gamma_A$ and $\Gamma_B$. As a sample application, note that it follows immediately from Gromov's link condition that $\xcplx$ is non\=/positively curved when $\Gamma_A$ and $\Gamma_B$ are flag simplicial complexes.
	
	The simplices $\sigma_A(v)$ and $\sigma_B(v)$ in \eqref{eq:intro} are determined using the $n$\=/colourings of $\Gamma_A$ and $\Gamma_B$.
	In fact, the vertices of $\xcplx$ are in a one\=/to\=/one correspondence with `pairs of simplices with complementary colours'. The simplices $\sigma_A$ and $\sigma_B$ are determined via this correspondence (this is explained in Section~\ref{sec:complex_X_Gamma}).

	Before continuing with the discussion of general properties of CLCCs, we wish to report that the class of fundamental groups of coupled link cube complexes is indeed fairly large. In fact, we prove:
	
	\begin{alphthm}\label{thm:intro.examples}
		The class of fundamental groups of connected non\=/positively curved CLCCs includes:
		\begin{itemize}
			\item every RAAG; 
			\item the commutator subgroup of every RACG;
			\item every orientable surface group;
			\item manifold and pseudo-manifold groups of arbitrarily high dimension. 
		\end{itemize}
	\end{alphthm}
	
	All the CLCCs needed to prove Theorem~\ref{thm:intro.examples} are defined in Section~\ref{sec:Examples}. More precisely, RAAGs and RACGs are obtained in Subsections~\ref{ssec:examples.RAAGs} and \ref{ssec:examples.RACGS}; surfaces in Subsection~\ref{ssec:examples.surfaces}; manifolds and pseudomanifolds in Subsections~\ref{ssec:examples.manifolds} and \ref{ssec:examples.pseudomanifolds}.
	
	\begin{rmk}
		The commutator subgroup of a RACG $\CC_\Gamma$ is a finite index subgroup $\CC_\Gamma$. In particular, every RACG is virtually the fundamental group of a CLCC. 
	\end{rmk}
	\begin{rmk}
		Surface groups and manifold groups are produced by constructing explicit cubulations of surfaces and PL manifold. In dimension greater than two, it is not clear what manifolds can be realised as CLCCs. A necessary condition is that the manifold is aspherical, but not much is known beyond that. It would be interesting to see what other properties are satisfied by these manifolds. For example, are they alway orientable?
	\end{rmk}

	\begin{rmk}
	 For any fixed dimension $n\geq 3$ we expect that CLCCs give rise to infinitely many non\=/isomorphic $n$\=/manifold groups. It is less clear whether they also give rise to infinitely many commensurability classes of such groups.
	\end{rmk}

	We begin our investigation of the geometry of CLCCs by showing that this construction is functorial from the category $\flag_n^2$ to $\npc$. Here, $\npc$ is the category with non\=/positively curved cube complexes as objects and cubical maps as morphisms, and $\flag_n$ is the category whose objects are $n$-coloured flag complexes and morphisms are simplicial maps that preserve the $n$-colouring. That is, in Theorem \ref{thm:functoriality.CAT0}, we prove the following:
	
	\begin{alphthm}\label{thm:intro.functoriality}
		A pair of colour\=/preserving maps of $n$\=/coloured flag simplicial complexes $f_A\colon\Gamma_A\to\Gamma_A'$ and $f_B\colon\Gamma_B\to\Gamma_B'$ induces a cubical map 
		\[
		\BX(f_A,f_B)\colon\xcplx\to\BX_{\Gamma_A',\Gamma_B'}.
		\]
		Furthermore, if $f_A$ and $f_B$ are inclusion of full subcomplexes, then $\BX(f_A, f_B)$ is a local isometry and hence induces an inclusion of fundamental groups.
	\end{alphthm}
	
	\begin{rmk}
		The CLCC construction can be applied also to simplicial complexes that are not flag. Theorem~\ref{thm:intro.functoriality} holds in this case as well with the only exception that $\BX(f_A, f_B)$ need not induce an inclusion of fundamental groups.
	\end{rmk}
	\begin{rmk}
		In the case that $\BX(f_A,f_B)$ is an inclusion of full subcomplexes, Theorem~\ref{thm:intro.functoriality} generalises the notion of parabolic subgroup of RAAGs and RACGs. 
	\end{rmk}
	
	Besides showing the naturality of the CLCC construction, Theorem~\ref{thm:intro.functoriality} is also helpful for studying the $\ZZ_2$\=/homology groups of $\xcplx$ in terms of the defining simplicial complexes. We prove the following (see Theorem~\ref{thm:chains.to.chains,cycles.to.cycles,amen} and Definition \ref{def:smartlypairedchains} for terminology used):
	
	\begin{alphprop}\label{prop:intro.homology.CLCC}
		If $\varOmega_A$ and $\varOmega_B$ are smartly paired non-zero cycles in $\Gamma_A$ and $\Gamma_B$ of dimension $d_A$ and $d_B$ respectively, then they define a non-zero $(d_A+d_B+2-n)$\=/cycle $\BX(\varOmega_A,\varOmega_B)$ in $\xcplx$.
	\end{alphprop}
	
	In the case that $\varOmega_A$ and $\varOmega_B$ are top dimensional, then the cycle $\BX(\varOmega_A,\varOmega_B)$ is non-trivial in $H_{d_A+d_B+2-n}(\xcplx; \ZZ_2)$. 
	This can be used to study the homological (and cohomological) dimension of CLCCs and their fundamental groups. One immediate consequence is the following.
	
	\begin{cor}
		If $\Gamma_A$ is $k$-dimensional with $H_k(\Gamma_A;\ZZ_2)\neq 0$ and $\Gamma_B$ is the simplicial join $\Ast_{i=1}^n B_i$, where every $B_i$ is a discrete set with at least $2$ elements, then $H_{k+1}\bigparen{\xcplx;\ZZ_2}\neq 0$.
	\end{cor}
	
	\
	
	The key advantage CLCCs have over RAAGs and RACGs when looking for exotic examples of groups is that CLCCs (and the links of their vertices) are not defined by a unique simplicial complex, but from a \emph{pair} of them. This can be very useful. For example, if one is looking for groups that are at the same time `nice' and `pathological', they can try to do it by encoding the `niceness' conditions in $\Gamma_A$ and the `pathological' conditions in $\Gamma_B$, so that the resulting cube complex $\xcplx$ enjoys both. Alternatively, one can also try to produce complexes $\Gamma_A$ and $\Gamma_B$ that are both `fairly nice' and `fairly pathological' and pair them in such a way that $\xcplx$ is `totally nice and pathological'.

	One instance where this approach of `spreading difficulties' among $\Gamma_A$ and $\Gamma_B$ proved fruitful is the theorem below (see Definition \ref{def:pairwise} for the notion of `pairwise 5-largeness').
	
	\begin{alphthm}\label{thm:intro.hyperbolic.CLCC}
		If $\Gamma_A, \Gamma_B$ are pairwise 5-large flag complexes, then the fundamental group of each connected component of $\xcplx$ is hyperbolic.
	\end{alphthm}
	
	We wish to remark that the notion of `pairwise 5-largeness' for a pair of flag complexes $\Gamma_A$ and $\Gamma_B$ does not require either of the two complexes to be $5$\=/large. 
	However, if either of the complexes $\Gamma_A$ or $\Gamma_B$ is 5-large, then the pair $\Gamma_A, \Gamma_B$ is pairwise 5-large. 
	As a consequence, this provides an alternative proof of the well\=/known fact that a right angled Coxeter group $\CC_\Gamma$ is hyperbolic if and only if $\Gamma$ is $5$\=/large .

	We obtain the following corollary from Theorems \ref{thm:4manifolds} and \ref{thm:3dimbarycentricsubdivision}.
	
	\begin{cor}
		Among connected non\=/positively curved CLCCs there are various cubulated $4$\=/manifolds with hyperbolic fundamental group.
	\end{cor}
	
	We show that in low dimensions there are no topological obstructions to a pair of simplicial complexes admitting a pairwise 5-large colouring. Namely, we prove the following in Theorems \ref{thm:barycentric.subdivision.hyperbolic} and \ref{thm:3dimbarycentricsubdivision}: 
	
	\begin{alphprop}
		Let $\Lambda_A, \Lambda_B$ be a pair of simplicial complexes of dimension $n\leq 3$. Then there exist simplicial complexes  $\Gamma_A$ and $\Gamma_B$ homeomorphic to $\Lambda_A$ and $\Lambda_B$ which have an $n+1$ colouring and are pairwise 5\=/large. 
	\end{alphprop}

	Together with Proposition~\ref{prop:intro.homology.CLCC} we obtain:
	
	\begin{cor}
		Connected non\=/positively curved CLCCs include a wealth of $4$\=/dimensional cube complexes whose fundamental groups are hyperbolic and of cohomological dimension $4$ (one such CLCC for every pair of $3$-dimensional flag simplicial complexes with non trivial third homology group. See Theorem~\ref{thm:3dimbarycentricsubdivision}).
	\end{cor}

	\begin{rmk}
		A rather different instance where the extra flexibility given by CLCCs paid off, is the work carried out by the first author in \cite{kropholler_hyperbolic}. There he uses the toolkit developed in this paper to produce various pathological groups. For example, he constructed hyperbolic groups having subgroups of type $FP_2$ that are not finitely presented.
	\end{rmk}
	
	Given the added flexibility of CLCCs it is worth asking several questions relating to proving new results and structural theorems for the class. 
	We outline some questions relating to the geometry at infinity here. 
	
	Firstly, in the class of RAAGs and RACGs, there are visual conditions for is 1-endedness. 
	RAAGs (RACGs resp.) are one ended if the defining graph is connected (has no separating clique resp.). 
	\begin{question}
		When is $\pi_1(\xcplx)$ 1-ended? 
	\end{question}

	Under the added assumption that the CLCC is CAT(0) (or hyperbolic), we can study its visual boundary. 
	The class of RACGs already provides interesting concrete examples of groups with Menger curve boundary \cite{genevievehaulmark}.
	It is also known \cite{januszkiewicz_hyperbolic_2003}, that higher dimensional spheres ($n>3$) do not appear as the boundary of hyperbolic RACGs. 
	Thus we end with the question of what boundaries occur in the class of CLCCs. 
	
	\begin{question}
	 What are the possible topological spaces that can appear as the CAT(0) or Gromov boundary of a CLCC? More Specifically,
		\begin{enumerate}
			\item Can we obtain higher dimensional Menger compacta?
			\item Can we obtain spheres of dimension $>3$ as Gromov boundaries?
			\item Can we obtain homology spheres of dimension $>3$ as Gromov boundaries?
	\end{enumerate} 
	\end{question}

	The second and third parts of the above questions relate to Questions~\ref{qu:hypmanifold} and \ref{qu:poincareduality}. 
	A deeper understanding of the Gromov boundary could also be used to differentiate CLCCs from RACGs and lattices in $\SO(n, 1)$. 
	At this time, we do not yet have explicit examples of hyperbolic CLCCs that are not commensurable to RACGs or lattices in $\SO(n, 1)$.
	However, we conjecture that such examples do exist. For instance, we suspect that Theorems~ \ref{thm:barycentric.subdivision.hyperbolic} and \ref{thm:3dimbarycentricsubdivision}  could be used to obtain such examples in low dimensions.

	\subsection{Organisation of the paper}
	In Section~\ref{sec:prelims} we give some notation and recall some facts about metric spaces and complexes. In Section~\ref{sec:complex_X_Gamma} we define CLCCs and we prove various general results about them, including Theorem \ref{thm:intro.functoriality} and Proposition~\ref{prop:intro.homology.CLCC} and a criterion for connectedness. In Section~\ref{sec:Examples} we give some examples to illustrate the flexibility of the CLCC construction.
	
	In Section \ref{sec:hyperbolic}, we study the combinatorics of hyperplanes in CLCCs and use these to show hyperbolicity in certain cases including Theorem \ref{thm:intro.hyperbolic.CLCC} and its corollaries.

	\subsection*{Acknowledgements} 
	It is our great pleasure to thank the anonymous referee for their thorough and helpful comments and for pointing us towards \cite{genevois}. Besides significantly improving the general exposition, their suggestions allowed us to provide a much more general and streamlined proof of the hyperbolicity criteria for CLCCs.
	
	Part of this work was undertaken at MSRI, Berkeley (during the program Geometric Group Theory), where research is supported by the National Science Foundation under Grant No. DMS-1440140. 
	It was also funded by the Deutsche Forschungsgemeinschaft (DFG, German Research Foundation) under Germany's Excellence Strategy EXC 2044 –390685587, Mathematics Münster: Dynamics–Geometry–Structure.

\section{Notation and preliminaries}
\label{sec:prelims}

The constructions in this paper will use the language of CAT(0) cube complexes.
We will now recall some well\=/known facts and definitions. Complete definitions and basic facts about CAT(0) geometry, can be found in \cite{bridson_metric_1999,sageev_cat0_2014}.

\subsection{Cube and simplicial complexes}\label{subsec:cubes}
A {\em cube complex} $X$ is a collection of cubes glued together via isometries of faces (for more details see \cite{sageev_cat0_2014}). A connected cube complex comes with a natural \emph{Euclidean metric} obtained by putting the Euclidean metric on each cube $c = [0, 1]^n$ and taking the resulting path metric. 
We will insist that the characteristic map of a cube $\chi_c\colon c\to X$ is injective. 
This can always be arranged by barycentric subdivision.

Similarly, a \emph{$\Delta$\=/complex} is a collection of simplices glued together via isometries of faces. We will once again insist that the characteristic map of a simplex $\chi_\sigma\colon \sigma\to X$ is injective. A $\Delta$\=/complex is \emph{simplicial} if every pair of simplices share at most one face. A simplex in a simplicial complex is uniquely determined by its vertices, we will therefore identify every simplex with its set of vertices $\sigma\leftrightarrow[z_0,\ldots,z_k]$. 

A map $f\colon X\to Y$ between cube complexes is a \emph{cubical map} if the image of every $n$\=/cube $c\subseteq X$ is a $k$\=/cube $c'\subseteq Y$ for some $k\leq n$ and the restriction of $f$ to $c$ is a projection $[0,1]^n\to [0,1]^k$. Similarly, a map between $\Delta$\=/complexes is \emph{simplicial} if it sends simplices to simplices (possibly collapsing them).

If $X$ is a cube complex and $c\subseteq X$ is a cube of dimension $k$, its {\em link} in $X$, denoted $\lk(c, X)$, is the $\Delta$\=/complex with an $(n-k-1)$-simplex for each $n$-cube containing $c$. These simplices are identified using the equivalence relation coming from inclusion of cubes. Namely, if $c$ is a face of $c'$, then the simplex corresponding to $c$ is a face of the simplex corresponding to $c'$.  Similarly, if $\sigma\subseteq\Lambda$ is a $k$\=/simplex in a $\Delta$\=/complex, its {\em link} $\lk(\sigma, \Lambda)$ is the $\Delta$\=/complex with an $(n-k-1)$-simplex for each $n$-simplex containing $\sigma$, with a similarly defined equivalence.

A $\Delta$\=/complex $\Lambda$ is {\em flag}, if it is simplicial and every set of pairwise adjacent vertices spans a simplex. 
Let $V\subsetneq V(\Lambda)$. 
Define the {\em full subcomplex} $\Lambda(V)$ to be the subcomplex of all simplices in $\Lambda$ whose vertices are in $V$. 
It follows that full subcomplexes of flag complexes are flag. 
If $\Lambda$ is a flag complex, then the link of a simplex $\sigma\subseteq\Lambda$ is equal to the full subcomplex of $\Lambda$ spanned by all vertices adjacent to the simplex i.e.
\[
\lk(\sigma,\Lambda)\coloneqq \Lambda \bigparen{ \{v\in\vertex(\Lambda) \bigmid v\notin \sigma, [v\cup\sigma] \mbox{ is a simplex of } \Lambda\} }.
\]

A cube complex is \emph{non\=/positively curved} if the link of every vertex is a flag complex. Equivalently, the universal cover supports a CAT(0) metric.

\subsection{Cohomological Dimension of Groups} 

We refer the reader to \cite{bieri_homological_1981, brown_cohomology_1982} for more details on (co)homology of groups. 

The {\em cohomological dimension} of a discrete group $G$ (denoted $cd(G)$), is the projective dimension of $\ZZ$ considered as the trivial $\ZZ G$ module. Equivalently, the cohomological dimension can be defined as 
\[
cd(G) = \max\bigbrace{n\bigmid H^n(G; M)\neq 0\mbox{ for some } \ZZ G \mbox{-module } M}. 
\] 
There are similar definitions for {\em homological dimension} (denoted $hd(G)$), where projective modules are replaced by flat modules and cohomology groups are replaced with homology groups. 

Since projective modules are flat, it is always the case that $hd(G)\leq cd(G)$. Thus, to show that a group has cohomological dimension $\geq n$ one just needs to find some $\ZZ G$-module $M$ such that $H_n(G; M)\neq 0$. For us, it will be convenient to use $\ZZ_2$\=/coefficients to find such homology groups. 
We will be interested in the case that $G$ is the fundamental group of a compact non-positively curved cube complex. 
In this case, $G$ is of type $FP$ and thus $hd(G) = cd(G)$ see \cite[pg.204, Ex. 6]{brown_cohomology_1982}

If a finite cube complex $X$ is non\=/positively curved then it is a classifying space for its fundamental group $G\coloneqq\pi_1(X)$. It follows that the homology group $H_n(G;\ZZ_2)$ is equal to $H_n(X;\ZZ_2)$. 
To show that fundamental groups of non\=/positively curved cube complexes have high cohomological dimension it is therefore enough to prove that said complexes have non\=/trivial high dimensional $\ZZ_2$\=/homology groups.

\begin{rmk}
For torsion free groups $cd(G)$ is invariant under commensurability while groups that contain elements of finite order always have infinite cohomological dimension. If a group $G$ is virtually torsion free, it is then convenient to define the {\em virtual cohomological dimension} of $G$ ($vcd(G)$) as the cohomological dimension of any finite index torsion free subgroup. By the above remark $vcd(G)$ equals $cd(G)$ when $G$ is torsion-free.
\end{rmk}

\subsection{Homology of cube complexes}
Computing the $\ZZ_2$\=/homology groups of a cube complex is a relatively simple task using cellular homology. Let $X$ be a cube complex. As a topological space, $X$ is naturally a CW complex. That is, $X$ can be constructed inductively: the $0$\=/skeleton $X^{(0)}$ is the set of vertices and the $n$\=/skeleton $X^{(n)}$ is obtained by gluing all the $n$\=/cubes to the $(n-1)$\=/skeleton (to be precise, when doing this construction we are implicitly choosing an orientation on each cube of $X$. This will not play any role in the current discussion as we are only working with $\ZZ_2$\=/coefficients). 

The space of \emph{$n$\=/chains} with $\ZZ_2$\=/coefficients is the $\ZZ_2$\=/module $C_n(X;\ZZ_2)$ of finite formal sums $\sum_c \alpha_c c$ where $c$ ranges among the $n$\=/cubes of $X$ and $\alpha_c\in \ZZ_2$.

For general CW complexes, the differential $d_n\colon C_n(X;\ZZ_2)\to C_{n-1}(X;\ZZ_2)$ is defined in terms of degrees of attaching maps. Since we are only considering cube complexes and $\ZZ_2$, we can express the differential by linearly extending the map given by
\[
 d_n(c)\coloneqq \smashoperator[r]{\sum_{e\text{ face of }c}} e
\]
where $c$ is an $n$\=/cube of $X$ and $e$ is an $(n-1)$\=/dimensional face of it.
We say that an $n$-chain $c$ is an {\em $n$-cycle}, if $d_n(c) = 0$. 

It is well known that $d_{n-1}\circ d_n=0$ and that the homology of the complex
\[
\begin{tikzcd}
  \cdots \arrow{r} & C_n \arrow{r}{d_n} & C_{n-1}\arrow{r}{d_{n-1}} 
      &C_{n-2}\arrow{r}{d_{n-2}} &\cdots
\end{tikzcd}
\]
is isomorphic to the singular homology of $X$ (see \emph{e.g.} \cite{hatcher_algebraic_2002}, Theorem 2.35).

Throughout the paper we will actually be working with reduced homology. That is, the homology of the augmented chain complex
\[
\begin{tikzcd}
  \cdots \arrow{r} & C_1 \arrow{r}{d_1} & C_{0}\arrow{r}{d_{0}} 
      &C_{-1}\cong \ZZ_2 
\end{tikzcd}
\]
where the augmentation map $d_{0}$ sends every $0$\=/cube (vertex) $v$ to $1\in \ZZ_2$. 
Such an augmentation can also be formalized by declaring that there exists a unique empty cube in $X$, and this empty cube gives rise to the group $C_{-1}$. 

Note that the discussion also above applies to $\Delta$\=/complexes.

\

Sometimes, chains and cycles can be recognised locally. 
Let $X$ be a cube complex and fix some $k\leq n\leq\dim(X)$. For every $k$\=/cube $e$ in $X$ we have a \emph{localisation map} $\loc_e\colon C_n(X;\ZZ_2)\to C_{n-k-1}\bigparen{\lk(e,X);\ZZ_2}$. Specifically, since each $(n-k-1)$\=/simplex $\sigma$ in $\lk(e,X)$ is given by an $n$\=/cube $c(\sigma)$ of $X$ containing $e$, it makes sense to define
\[
 \loc_e\Bigparen{\sum\alpha_c c}\coloneqq \sum_{\sigma\in \lk(e,X)}\alpha_{c(\sigma)} \sigma.
\]
In the case $k=n$, the link is going to be empty, and therefore the only non\=/trivial module in the chain complex of $\lk(c,X)$ is $C_{-1}$. When this happens, the localisation map $\loc_c$ will send the chain $\sum \alpha_c c$ to $\sum\alpha_c\in\ZZ_2\cong C_{-1}$.

We can define a localisation map $\loc_\tau\colon C_n(\Lambda;\ZZ_2)\to C_{n-k-1}\bigparen{\lk(\tau,\Lambda);\ZZ_2}$ for $k$\=/simplices $\tau$ in a $\Delta$\=/complex $\Lambda$ just as we did for cubes in cube complexes, and we have the following:

\begin{lem}\label{lem:n.chains.recognised.locally}
 Fix some $k<n$. Then an $n$\=/chain $\varSigma$ in a cube complex $X$ (resp. in a $\Delta$\=/complex $\Lambda$) is a cycle if and only if $\loc_e\paren{\varSigma}$ (resp. $\loc_\tau\paren{\varSigma}$) is a cycle for every $k$\=/cube $e\subsetneq X$ (resp. $\tau$ is a $k$-simplex of $\Lambda$). 
 
 Moreover, if $\varSigma$ is a boundary then so is $\loc_e\paren{\varSigma}$ for every $k$\=/cube $e$.
\end{lem}
\begin{proof}
 We prove the lemma for cube complexes, as the same proof works for $\Delta$\=/complexes.
 It is immediate to check that we have the equality $d_{n-k-1}\bigparen{\loc_e(\varSigma)}=\loc_e\bigparen{d_n(\varSigma)}$. It is hence clear that $\loc_e$ sends cycles to cycles and boundaries to boundaries. It only remains to show that if $d_{n-k-1}\bigparen{\loc_e(\varSigma)}=0$ for every $k$\=/cube $e\in X$, then $d_n(\varSigma)=0$. 
 
 As $d_{n-k-1}\bigparen{\loc_e(\varSigma)}=\loc_e\bigparen{d_n(\varSigma)}$, it will be enough to show that if $\varOmega\in C_{n-1}(X;\ZZ_2)$ is a non trivial $(n-1)$\=/chain then there exists a $k$\=/cube $e$ such that $\loc_e(\varOmega)$ is not trivial. This is readily done: let $\varOmega=\sum \alpha_c c$ and choose $c$ so that $\alpha_c$ is not $0$. If $k=n-1$, by choosing $e=c$ we immediately deduce that $\loc_e(\varOmega)=\alpha_c\neq 0$. Otherwise, fix any $k$\=/cube $e\subsetneq c$: the cube $c$ will identify an $(n-k-1)$\=/simplex $\sigma$ in $\lk(e,X)$ with coefficient $\alpha_c$ in $\loc_e(\varOmega)$. In particular, $\loc_e(\varOmega)$ is not trivial.
\end{proof}

\begin{rmk}\label{rmk.to.recognise.chains.locally.few.simplices.suffice}
 For the proof of the fact that if all the localisations of an $n$-chain $\varSigma$ are cycles then $\varSigma$ is a cycle as well, we did not really need to know that $\loc_e(\varSigma)$ is a cycle for \emph{every} $k$\=/cube $e$. Indeed, it is actually enough to know that for every $(n-1)$\=/cube $c$ there exists a $k$\=/cube $e\subseteq c$ such that $\loc_e(\varSigma)$ is a cycle.
\end{rmk}

\subsection{Simplicial joins and homology.} The {\em simplicial join} $\Lambda\ast \Delta$ of two simplicial complexes is the simplicial complex with vertex set $\Lambda^{(0)}\sqcup \Delta^{(0)}$. The vertices $z_0, \dots, z_n$ span a simplex if and only if the sets $Z_\Lambda = \bigbrace{z_i \bigmid z_i\in \Lambda^{(0)}}$ and $Z_\Delta = \bigbrace{z_i \bigmid z_i\in \Delta^{(0)}}$ span simplices in $\Lambda$ and $\Delta$ respectively (the join of $\Lambda$ with an empty complex is equal to $\Lambda$ itself). 

The simplices of $\Lambda\ast \Delta$ are given by the joins of (possibly empty) simplices in $\Lambda$ with (possibly empty) simplices in $\Delta$. Therefore, given a $k$\=/chain $\varSigma=\sum\alpha_\sigma\sigma$ in $\Lambda$ and an $l$\=/chain $\varOmega=\sum\beta_\tau\tau$ in $\Delta$, we can define a $(k+l+1)$\=/chain $\varSigma\ast\varOmega$ in $\Lambda\ast \Delta$ by summing over all the joins of all simplices in $\varSigma$ with all simplices in $\varOmega$. That is, we set
\[
 \varSigma\ast\varOmega\coloneqq \sum_{\sigma,\tau}\alpha_\sigma\beta_\tau \paren{\sigma\ast\tau}.
\]
Note that the above expression is bilinear, and hence it gives us a linear map $C_k(\Lambda;\ZZ_2)\otimes C_l(\Delta;\ZZ_2)\to C_{k+l+1}(\Lambda\ast \Delta;\ZZ_2)$. 

In the case that $k = -1$ we have a unique $-1$ simplex $\sigma$ of $\Lambda$. 
Thus given an $l$-simplex $\tau$ of $\Delta$ we define $\sigma\ast \tau$ to be $\tau$. Thus the linear map $C_{-1}(\Lambda;\ZZ_2)\otimes C_l(\Delta;\ZZ_2)\to C_{l}(\Lambda\ast \Delta;\ZZ_2)$ is the map induced by the inclusion $\Delta\to \Lambda\ast\Delta$. 

\begin{lem}\label{lem:joinsofcycles}
  Let $\varSigma$ be a non\=/trivial $k$\=/chain in $\Lambda$ and $\varOmega$ a non\=/trivial $l$\=/chain in $\Delta$. Then $\varSigma\ast \varOmega$ is a $(k+l+1)$\=/cycle in $\Lambda\ast \Delta$ if and only if $\varSigma$ and $\varOmega$ are cycles as well. 
\end{lem}
\begin{proof}
  The boundary of a simplex $\sigma\ast \tau$ is $(d_k \sigma\ast \tau) + (\sigma\ast d_l \tau)$. Since $\ast$ is bilinear on the space of chains, it is easy to check that $d_{k+l+1}\paren{\varSigma\ast\varOmega}=\paren{d_{k}\varSigma\ast \varOmega}+\paren{\varSigma\ast d_{l}\varOmega}$. It follows that if $\varSigma$ and $\varOmega$ are cycles, then $\varSigma\ast\varOmega$ is a cycle as well.
  
  Vice versa, if $\paren{d_{k}\varSigma\ast \varOmega}+\paren{\varSigma\ast d_{l}\varOmega}=0$ then both $d_{k}\varSigma\ast \varOmega$ and $\varSigma\ast d_{l}\varOmega$ must be $0$ because their supports are disjoint (on one side we have joins of $(k-1)$\=/simplices with $l$\=/simplices, on the other joins of $k$\=/simplices with $(l-1)$\=/simplices).
  
  We now have $d_{k}\varSigma\ast \varOmega=0$. Let $\varOmega=\sum\beta_\tau\tau$; since $\varOmega$ is not trivial, there exists a simplex $\tau$ so that $\beta_\tau\neq 0$. We deduce that $d_k\varSigma\ast\tau=0$ and hence $d_k\varSigma=0$. A symmetric argument implies that $d_l\varOmega=0$ as well.
\end{proof}

\section{Coupled link cube complexes}
\label{sec:complex_X_Gamma}

\subsection{Notation}
Throughout the paper we will use the following convention:
\begin{convention}
 A graph is \emph{$n$\=/coloured} if it is $n$\=/partite and there is a fixed $1$\=/to\=/$1$ correspondence between the partition sets and $\{1,\ldots,n\}$. In particular, every $n$\=/partite graph can be $n$\=/coloured in $n!$ different way respecting the $n$\=/partite structure. A simplicial complex is $n$\=/partite or $n$\=/coloured if its $1$\=/skeleton is. A priori, we do not insist that all the partitioning sets be non empty.
\end{convention}

Let $A_1,\ldots,A_n$ and $B_1,\ldots,B_n$ be (possibly empty) finite sets. Then the joins $\BA\coloneqq A_1\ast\cdots*A_n$ and $\BB\coloneqq B_1*\cdots *B_n$ are $n$\=/coloured simplicial complexes of dimension (at most) $n-1$. We will denote their simplices with fraktur letters. We say that a simplex $\fka\subseteq \BA$ (resp. $\fkb\subseteq\BB$) has a \emph{vertex on the $i$\=/th coordinate} if one of its vertices is in $A_i$ (resp $B_i$). Two simplices $\fka\subseteq \BA$ and $\fkb\subseteq \BB$ are \emph{complementary} if for every $i=1,\ldots,n$ exactly one between $\fka$ and $\fkb$ has a vertex on the $i$\=/th coordinate. Given a simplex $\fka\subseteq\BA$ with a vertex $a_i$ on the $i$\=/th coordinate, we will denote by $\fka_i$ the subsimplex obtained removing $a_i$. Conversely, if $\fka$ does not have a vertex on the $i$\=/th coordinate we denote by $\fka^i$ a simplex obtained by adding to $\fka$ an element $a_i\in A_i$ as a vertex: $\fka^i=\fka*a_i$. 
Note that $\fka^i$ is not uniquely determined by $\fka$ as it also depends on the choice of $a_i$. In particular, $(\fka_i)^i$ might be different from $(\fka^i)_i$ (the latter is always equal to $\fka$). On the contrary, when adding or removing vertices on different coordinates the result does not depend on the order of such operations; we will thus omit the parentheses and write $\fka_i^j\coloneqq(\fka^j)_i=(\fka_i)^j$, $\fka_{ij}\coloneqq(\fka_j)_i=(\fka_i)_j$ (see Figure~\ref{fig:notation.simplices}).

\begin{figure}
    \centering
    \includegraphics{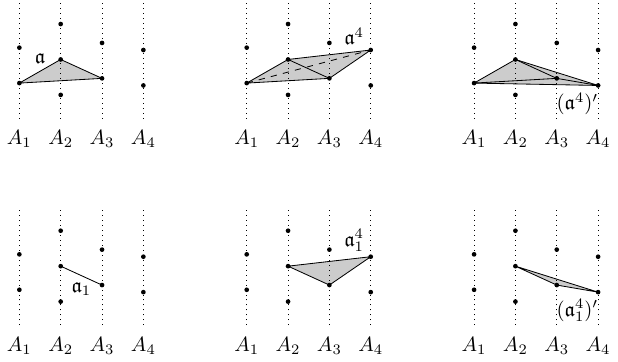}
    \caption{Simplices obtained adding and/or removing a coordinate.}
    \label{fig:notation.simplices}
\end{figure}

Note that if $(\fka,\fkb)$ is a pair of complementary simplices, then so are $(\fka_i,\fkb^i)$ and $(\fka^j,\fkb_j)$.

\

For every $i=1,\ldots,n$ the join $A_i*B_i$ is a complete bipartite graph. The Cartesian product 
\[
\BK\coloneqq (A_1*B_1)\times\cdots\times(A_n*B_n)
\]
is an cube complex of dimension $\leq n$. The set of vertices of $\BK$ is the product $(A_1\sqcup B_1)\times \cdots \times (A_n\sqcup B_n)$ \emph{i.e.} every vertex $v\in\BK$ is given by a choice of $n$ coordinates $v=(x_1,\ldots,x_n)$ where the $i$\=/th coordinate $x_i$ is either an element of $A_i$ (\emph{A-coordinate}) or of $B_i$ (\emph{B-coordinate}). The sets of $A$\=/coordinates and $B$\=/coordinates determine two complementary simplices of $\BA$ and $\BB$. This yields a 1\=/to\=/1 correspondence between the vertices of $\BK$ and the pairs of complementary simplices $(\fka,\fkb)$ in $\BA\times\BB$; we will therefore denote the vertices of $\BK$ \emph{via} their \emph{coordinate simplices} $(\fka,\fkb)$.

Using the above conventions, a vertex $(\fka,\fkb)\in\BK$ is joined by an edge of $\BK$ with exactly the vertices of the form $(\fka_i,\fkb^i)$ or $(\fka^i,\fkb_i)$ for some $i=1,\ldots,n$. We say that such an edge is an \emph{edge in the $i$\=/th coordinate}. 
Note that two edges meeting at a vertex $(\fka,\fkb)\in\BK$ lie in a common square of $\BK$ if and only if they are edges in different coordinates $i\neq j$. When this is the case, we denote by $(\fka',\fkb')$ the fourth vertex of such square. The vertex $(\fka',\fkb')$ is the unique vertex obtained from $(\fka,\fkb)$ by changing both the $i$\=/th and the $j$\=/th coordinate accordingly. Moreover, the square containing both $(\fka,\fkb)$ and $(\fka',\fkb')$ is unique.

In general, the $k$\=/cubes of $\BK$ are the products of $k$ edges in $k$ different coordinates. Moreover, given any vertex $(\fka,\fkb)\in\BK$ and another vertex $(\fka',\fkb')$ satisfying the following equality
\[
(\fka',\fkb')=\left(\fka_{i_{h+1},\ldots,i_k}^{i_1,\ldots,i_h},\fkb_{i_1,\ldots,i_h}^{i_{h+1},\ldots,i_k}\right) 
\text{   with }h\leq k \text{ and }i_s\neq i_t\ \forall 1\leq s< t\leq k,
\]
that is $(\fka', \fkb')$ differs from $(\fka,\fkb)$ in $k$ coordinates. 
Then there exists a unique $k$\=/cube of $\BK$ containing the vertices $(\fka,\fkb)$ and $(\fka',\fkb')$. 

If in the above equation we have $h = 0$, that is $\fka'=\fka_{i_1,\ldots,i_k}$ (and hence $\fkb'=\fkb^{i_1,\ldots,i_k}$), then the $k$\=/cube containing $(\fka,\fkb)$ and $(\fka',\fkb')$ is uniquely determined by $\fka$ and $\fkb'$ because its vertices are those whose coordinate simplices are subsimplices of $\fka$ and $\fkb'$. In this case we say that $(\fka,\fkb)$ and $(\fka',\fkb')$ are \emph{extremal} vertices and we denote the $k$\=/cube containing them by $Q(\fka,\fkb')$. Every cube in $\BK$ has a unique pair of extremal vertices and therefore it can be uniquely expressed as $Q(\fka,\fkb)$ for the appropriate $\fka\subseteq\BA$ and $\fkb\subseteq\BB$. 
If $\fka$ and $\fkb$ share $k$ colours, then $Q(\fka,\fkb)$ is a $k$-cube of $\BK$. 

It can be convenient to think of the edges of $\BK$ as oriented. We use the convention that edges point in the direction of the vertex with fewer $A$\=/coordinates \emph{i.e.} an edge is oriented from $(\fka,\fkb)$ to $(\fka_i,\fkb^i)$. The extremal vertices of a $k$\=/cube of $\BK$ are the initial and terminal vertices of all oriented paths of length $k$ in its $1$\=/skeleton (see Figure~\ref{fig:orientation_edges_k_cubes}).

\begin{figure}
    \centering
    \begin{subfigure}{0.3 \textwidth}
	\includegraphics{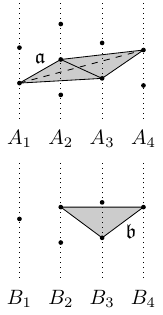}     
    \end{subfigure}
    \begin{subfigure}{0.6 \textwidth}
	\includegraphics{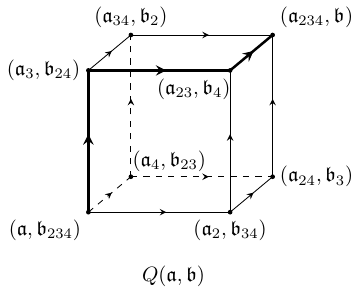}     
    \end{subfigure}
    \caption{A $3$\=/cube determined by $4$\=/coloured simplices overlapping in the coordinates $2,3$ and $4$. The edges are oriented towards vertices with fewer $A$\=/coordinates and a path between extremal vertices is highlighted.}
    \label{fig:orientation_edges_k_cubes}
\end{figure}

\subsection{Construction of coupled link cube complexes}
Let $\Gamma_A$ and $\Gamma_B$ be two $n$\=/coloured simplicial complexes and let $A_1,\ldots,A_n$ and $B_1,\ldots,B_n$ be the partitioning sets (where the $i$\=/th partitioning set is the one corresponding to the `colour' $i$). 
Thus we have inclusion of subcomplexes, $\Gamma_A\subseteq A_1\ast\dots\ast A_n$ and $\Gamma_B\subseteq B_1\ast\dots\ast B_n$.

\begin{de}
 The \emph{coupled link cube complex (CLCC)} associated to $\Gamma_A$ and $\Gamma_B$ is the full subcomplex $\xcplx\subseteq\BK$ whose vertices are the vertices of $\BK$ with (possibly empty) coordinate simplices in $\Gamma_A$ and $\Gamma_B$:
\[
{\rm V}\bigparen{\xcplx}=\bigbrace{(\fka,\fkb)\in\BK\bigmid \fka\subseteq\Gamma_A,\ \fkb\subseteq\Gamma_B}.
\]
\end{de}

If $\xcplx$ contains two extremal vertices $(\fka,\fkb_{i_1,\ldots,i_k})$ and $(\fka_{i_1,\ldots,i_k},\fkb)$, then $\xcplx$ also contains the $k$\=/cube $Q(\fka,\fkb)$. Indeed, the coordinate simplices of the vertices of $Q(\fka,\fkb)$ are contained in $\fka$ and $\fkb$ and are therefore in $\Gamma_A$ and $\Gamma_B$. In particular, every oriented path of length $k$ in the one skeleton of $\xcplx$ lies on the boundary of a (unique) $k$\=/cube of $\xcplx$. 
We see from this description that $\xcplx$ is equal to the union of $Q(\fka, \fkb)$ as $\fka$ (resp. $\fkb$) ranges over (possibly empty) simplices of $\Gamma_A$ (resp. $\Gamma_B$). 

\begin{rmk}
The complex $\xcplx$ is non\=/empty if and only if there exists a pair of complementary simplices in the pair $\Gamma_A, \Gamma_B$. In particular, we must have $\dim(\Gamma_A)+\dim(\Gamma_B)\geq n-2$. In general, we have:
\[
\dim\bigparen{\xcplx}\leq \dim(\Gamma_A)+\dim(\Gamma_B)+2-n 
\]
(with the convention $\dim(\emptyset)=-1$).
\end{rmk}

It will be useful for what follows to define $\lk(\emptyset, \Gamma) = \Gamma$. 

The motivating reason for defining coupled link cube complexes is the following:

\begin{lem}\label{lem:links.in.xcplx}
The link of a cube $Q(\fka,\fkb)$ in $\xcplx$ is the join of the links of $\fka\subseteq \Gamma_A$ and $\fkb\subseteq \Gamma_B$:
\[
\lk\bigparen{Q(\fka,\fkb),\xcplx}=\lk(\fka,\Gamma_A)*\lk(\fkb,\Gamma_B).
\]
\end{lem} 

\begin{proof}
 Let $k=\dim(Q(\fka,\fkb))$. Then a vertex in the link of $Q(\fka,\fkb)$ is a cube $C\subseteq \xcplx$ with $\dim(C)=k+1$ and $Q(\fka,\fkb)\subsetneq C$. Such a cube must be of the form $Q(\fka^j,\fkb)$ or $Q(\fka,\fkb^j)$. Any $\fka^j$ determines a vertex in $\lk(\fka,\Gamma_A)$ and any $\fkb^j$ determines a vertex in $\lk(\fkb,\Gamma_B)$. This yields a bijection between the vertices of the link of $Q(\fka,\fkb)$ to the union $\rm{V}\bigparen{\lk(\fka,\Gamma_A)}\sqcup\rm{V}\bigparen{\lk(\fkb,\Gamma_B)}$, the latter being precisely the set of vertices of $\lk(\fka,\Gamma_A)*\lk(\fkb,\Gamma_B)$. We need to show that this bijection induces an isomorphism of simplicial complexes.

Let $C,C'$ be two vertices in $\lk\bigparen{Q(\fka,\fkb),\xcplx}$. They are joined by an edge if and only if there exists a $(k+2)$\=/dimensional cube $C''\subseteq\xcplx$ with $C\cup C'\subsetneq C''$. There are three possible cases:
\begin{enumerate}[a)]
\item $C=Q(\fka^j,\fkb)$ and $C'=Q(\fka^k,\fkb)$,
\item $C=Q(\fka^j,\fkb)$ and $C'=Q(\fka,\fkb^k)$,
\item $C=Q(\fka,\fkb^j)$ and $C'=Q(\fka,\fkb^k)$.
\end{enumerate}
In all cases, it follows from our general discussion that such a $C''$ must take the form of $Q(\fka^{jk},\fkb)$, $Q(\fka^j,\fkb^k)$ or $Q(\fka,\fkb^{jk})$ respectively. Hence $C''$ exists if and only if the appropriate coordinate simplices exist (in particular we must have $j\neq k$). 

In the first case, it follows that there is such a $C''$ if and only if the appropriate simplex $\fka^{jk}$ exists, and this is equivalent to saying that $\fka^j$ and $\fka^k$ are joined by an edge in $\lk(\fka,\Gamma_A)$. The third case is analogous. The second case is even simpler: since $j$ and $k$ are different the cube $Q(\fka^j,\fkb^k)$ always exists, and hence we conclude, by the definition of join, that each vertex in $\lk(\fka,\Gamma_A)$ is linked to each vertex in $\lk(\fkb,\Gamma_B)$ by an edge.

We thus proved that the bijection between vertices extend to an isomorphism between the $1$\=/skeleta. All that remains to show is that a complete clique of $d+1$ vertices in $\lk\bigparen{Q(\fka,\fkb),\xcplx}$ is filled by a $d\=/simplex$ if and only if the corresponding simplex is in $\lk(\fka,\Gamma_A)*\lk(\fkb,\Gamma_B)$.

Let $Q(\fka^{i_0},\fkb),\ldots,Q(\fka^{i_l},\fkb),Q(\fka,\fkb^{i_{l+1}})\ldots,Q(\fka,\fkb^{i_d})$ be such a clique. Since we have edges between all those vertices, the indices $i_j$ must be all different. As above, we deduce that there exists the appropriate $d$\=/simplex in $\lk\bigparen{Q(\fka,\fkb),\xcplx}$ if and only if the simplices $\fka^{i_0\cdots i_l}$ and $\fkb^{i_{l+1}\cdots i_d}$ are in $\Gamma_A$ and $\Gamma_B$, and hence if and only if there exists the appropriate $d$\=/simplex in $\lk(\fka,\Gamma_A)*\lk(\fkb,\Gamma_B)$.
\end{proof}

Gromov's characterisation of non\=/positively curved cube complexes immediately implies the following.

\begin{cor}\label{cor:flag.implies.cat0.CLCC}
If the simplicial complexes $\Gamma_A$ and $\Gamma_B$ are flag then $\xcplx$ is a (possibly empty) non\=/positively curved cube complex.
\end{cor}

The converse of the above corollary is not true. For example, let $\Gamma_A$ be obtained from $\{a_+,a_-\}\ast\{a_+,a_-\}\ast\{a_+,a_-\}$ by removing the simplex $[a_-,a_-,a_-]$ and let $\Gamma_B$ have a single vertex per coordinate and no edges. Then $\xcplx$ is a connected graph (and therefore non\=/positively curved) but $\Gamma_A$ is not flag. 

Recall that a cube (simplicial) complex has \emph{pure dimension} $d$ if it has dimension $d$ and every cube (simplex) is a face of a $d$\=/dimensional cube (simplex). We have the following:

\begin{lem}\label{lem:dimension.xcplx.smartly.paired}
 If $\Gamma_A$ and $\Gamma_B$ are $n$\=/coloured simplicial complexes of pure dimension $k_A$ and $k_B$ respectively and $\xcplx$ is not empty, then it is a cube complex of pure dimension $k_A+k_B+2-n$.
\end{lem}

\begin{proof}
 We already remarked that $\dim(\xcplx)$ is at most $d=k_A+k_B+2-n$. Let $Q(\fka,\fkb)$ be any cube of $\xcplx$. 
 Since $\Gamma_A$ has pure dimension $k_A$, there exists a $k_A$\=/dimensional simplex $\bar\fka\subseteq\Gamma_A$ containing $\fka$ and similarly there exists a $k_B$\=/dimensional simplex $\bar\fkb\subseteq\Gamma_B$ containing $\fkb$. It follows that $\bar\fka$ and $\bar\fkb$ overlap in $d$ coordinates and hence $Q(\bar\fka,\bar\fkb)$ is a $d$\=/dimensional cube of $\xcplx$ containing $Q(\fka,\fkb)$.
\end{proof}

We conclude this subsection remarking that, in general, the simplicial complexes $\Gamma_A$ and $\Gamma_B$ might contain some `junk' simplices that do not contribute in any way to the construction of the complex $\xcplx$ because they do not have complementary simplices. 

\begin{de}
 Let $\Gamma_A$, $\Gamma_B$ be $n$\=/coloured simplicial complexes. A simplex in $\fka\subseteq\Gamma_A$ (resp. $\fkb\subseteq\Gamma_B$) is \emph{junk} if it is not contained into any simplex $\fka'$ admitting a complementary simplex in $\Gamma_B$ (resp. $\fkb'$ with complementary in $\Gamma_A$). The complexes $\Gamma_A$ and $\Gamma_B$ are \emph{smartly paired} if they contain no junk simplices.
\end{de}

Note that $\Gamma_A$ and $\Gamma_B$ are smartly paired if no maximal simplex in $\Gamma_A$, $\Gamma_B$ is junk. Given any pair of simplicial complexes $\Gamma_A$ and $\Gamma_B$, one can produce a pair of smartly paired complexes yielding the same coupled link cube complex by recursively removing maximal junk simplices. 

\begin{rmk}
 The procedure of removing maximal junk simplices need not preserve flagness. For example, let $\Gamma_A$ be four $2$\=/simplices arranged to form an equilateral triangle as in Figure~\ref{fig:not smartly paired}. Colour the vertices of the innermost triangle with numbers 1,2,3, and colour the three exterior vertices with colour 4. Let $\Gamma_B$ be a single $2$\=/simplex with vertices of coulour 1,2,3. The innermost triangle in $\Gamma_A$ is the only junk simplex and removing it makes $\Gamma_A$ not flag. 
\end{rmk}

\begin{figure}
 \begin{tikzpicture}[draw=black]
  \draw [fill= black!10](0,0) node[below left]{4} -- 
    (2,0)node[below right]{4} -- 
    ++(120:2)node[above]{4} --cycle;
  \draw [postaction={pattern=flexible hatch,
        hatch distance=5pt,
        hatch thickness=0.2pt,
        pattern color=black}](60:1)node[left]{3}--
    ++(1,0)node[right]{2} --
    (1,0)node[below]{1}-- cycle;
   \path (-0.6, 0.8)node{$\Gamma_A=$};
   \path (4, 0.8)node{$\Gamma_B=$};
  \begin{scope}[rotate=180, shift={(-6.4,-1.3)}]
  \draw[fill= black!10] (60:1)node[right]{3}--
    ++(1,0)node[left]{2} --
    (1,0)node[above]{1}-- cycle;   
  \end{scope}
 \end{tikzpicture}
 \caption{Two flag simplicial complexes that are not smartly paired. Making them smartly paired makes $\Gamma_A$ not flag.}
 \label{fig:not smartly paired}
\end{figure}

Non\=/empty smartly paired simplicial complexes give rise to non\=/empty CLCCs. Therefore, Lemma~\ref{lem:dimension.xcplx.smartly.paired} implies

\begin{cor}\label{cor:dimension.xcplx.smartly.paired}
 If $\Gamma_A$ and $\Gamma_B$ are smartly paired $n$\=/coloured simplicial complexes of pure dimension $k_A$ and $k_B$ respectively, then $\xcplx$ is a cube complex of pure dimension $k_A+k_B+2-n$. 
\end{cor}

\subsection{Functoriality}
Let $\simp_n$ be the {\em category of $n$-coloured simplicial complexes}. That is, the objects of $\simp_n$ are $n$\=/coloured simplicial complexes and the morphisms are simplicial maps that preserve the colouring (\emph{i.e.} they send vertices of the $i$\=/th colour to vertices of the $i$\=/th colour). Let $\simp_n^2$ denote the product category $\simp_n\times\simp_n$. 

Let also $\cc$ be the {\em category of cube complexes}, which has cube complexes as objects and cubical maps as morphisms. 

\begin{lem}\label{lem:functoriality}
	The map sending the pair $(\Gamma_A, \Gamma_B)$ to $\xcplx$ defines a functor $\BX\colon\simp_n^2\to \cc$. If two morphisms $f_A\colon\Gamma_A\to\Gamma_A'$ and $f_B\colon\Gamma_B\to\Gamma_B'$ are injective (resp. surjective) then $\BX(f_A,f_B)\colon\xcplx\to\BX_{\Gamma_A',\Gamma_B'}$ is injective (resp. surjective).	
\end{lem}
\begin{proof}
  We need to define the image of a morphism under $\BX$. Let $(f, g)\colon (\Gamma_A, \Gamma_B)\to (\Gamma_A', \Gamma_B')$ be a pair of maps of $n$\=/coloured simplicial complexes an let $(\fka,\fkb)$ be any vertex of $\xcplx$. 
  
  Since the maps $f_A, f_B$ are simplicial and respect the $n$-coloured structure, the pair $\paren{f(\fka),f(\fkb)}$ denotes a vertex of ${\BX_{\Gamma_A',\Gamma_B'}}$ and it is straightforward to check that this map on vertices extends to a cubical map.
  
  If $f_A$ and $f_B$ are injective, then $\BX(f_A,f_B)$ is injective on the vertices, and hence injective. If $f_A$ and $f_B$ are surjective, then $\BX(f_A,f_B)$ is surjective on the vertices and, since the cube complex $\xcplx$ is defined as a full subcomplex of the product of the joins $A_i\ast B_i$, it is simple to check that $\BX(f_A,f_B)$ is surjective.
\end{proof}

In the above, the cubical maps are in fact injective on each cube. 

Before stating the next lemma, we should recall that a cube complex admits a natural metric obtained by identifying each $k$-cube with $[0, 1]^k\subset \RR^k$. With respect to these metrics we have:
  
\begin{lem}\label{lem:local isometry}
    If two morphisms $f_A$ and $f_B$ of $\simp_n$ are inclusions of full subcomplexes then $\BX(f_A,f_B)$ is a local isometry. 
\end{lem}
\begin{proof} 
As $f_A$ and $f_B$ are inclusions, by Lemma~\ref{lem:functoriality} we see that $\BX(f_A, f_B)$ is as well. 
It can be shown that an injective cubical map is a local isometry if and only if the induced map on links of vertices is an inclusion of a full subcomplex.
 \footnote{%
 	One way to make a formal argument is by equipping the links with an all\=/right spherical metric with maximum distance $\pi$ and remark that an injection between simplicial complexes equipped with their all\=/right spherical metric is isometric if and only if it is the embedding of a full subcomplex. The statement about cube complexes follows from the fact that every point has a neighbourhood that is isometric to a cone over its link (see the proof of \cite[Theorem I.7.16]{bridson_metric_1999}). The condition that the link of vertices embed isometrically implies that also said cones embed isometrically.
 For the non-positively curved case see \cite[Lemma 2.11]{haglund_special_2008}.
}
 
 By Lemma~\ref{lem:links.in.xcplx}, the link of a vertex cube $Q(\fka,\fkb)\in\xcplx$ is $\lk(\fka, \Gamma_A)\ast\lk(\fkb, \Gamma_B)$, while the link of its image in $\BX_{\Gamma_A',\Gamma_B'}$ is $\lk(f_A(\fka), \Gamma_A')\ast\lk(f_B(\fkb), \Gamma_B')$. Note that the map induced from $\BX(f_A,f_B)$ on these links coincides with the inclusion
  \[
  \lk\bigparen{f_A(\fka), f_A(\Gamma_A)}\ast\lk\bigparen{f_B(\fkb), f_B(\Gamma_B)} 
  \hookrightarrow \lk(f_A(\fka), \Gamma_A')\ast\lk(f_B(\fkb), \Gamma_B').  
  \]
  This is the inclusion of a full subcomplex since $f_A(\Gamma_A)$ and $f_B(\Gamma_B)$ are full subcomplexes. 
\end{proof}

Let also $\flag_n\subsetneq\simp_n$ be the category of \emph{flag} $n$\=/coloured simplicial complexes and $\npc\subsetneq\cc$ be the category of non\=/positively curved cube complexes. Then we have the following:

\begin{thm}\label{thm:functoriality.CAT0}
  The functor $\BX\colon\simp_n^2\to\cc$ restricts to a functor $\BX\colon\flag_n^2\to\npc$. Moreover, if $\Gamma_A', \Gamma_B'$ are flag and $f_A\colon\Gamma_A\to\Gamma_A'$ and $f_B\colon\Gamma_B\to\Gamma_B'$ are inclusions of full subcomplexes, then $\BX(f_A,f_B)\colon\xcplx\to\BX_{\Gamma_A',\Gamma_B'}$ induces an inclusion of their fundamental groups.
\end{thm}
\begin{proof}
  The fact that $\BX$ restricts to a functor to $\npc$ follows from Lemma~\ref{lem:functoriality} and Corollary~\ref{cor:flag.implies.cat0.CLCC}.
  
  If $f_A$ and $f_B$ are inclusions of full subcomplexes, it follows from Lemma~\ref{lem:local isometry} that $\BX(f_A,f_B)\colon\xcplx\to\BX_{\Gamma_A',\Gamma_B'}$ is a local isometry. Since $\xcplx$ and $\BX_{\Gamma_A',\Gamma_B'}$ are non\=/positively curved spaces, the map induced on their fundamental groups 
  \[
  \BX(f_A,f_B)_*\colon\pi_1\bigparen{\xcplx,x_0}\to\pi_1\bigparen{\BX_{\Gamma_A',\Gamma_B'},\BX(f_A,f_B)(x_0)}
  \]
  is injective (see \cite[Proposition II.4.14]{bridson_metric_1999}).
\end{proof}

\subsection{Connectedness}
In general, it is not trivial to verify if the cube complex $\xcplx$ is connected. The following construction aims to simplify this task. We define a \emph{connection graph} $\CG_A$ whose vertices are the vertices of $\xcplx$ with maximal $A$\=/coordinates
\[
\rm{V}\bigparen{\CG_A}\coloneqq\bigbrace{(\bar\fka,\ubar\fkb)\in\xcplx \bigmid \bar\fka\subseteq\Gamma_A\text{ maximal simplex}}
\]
(note that the maximal simplices do not need to have the same dimension in general).
 We connect two vertices $(\bar\fka,\ubar\fkb)$ and $(\bar\fka',\ubar\fkb')$ in with an edge in $\CG_A$ (denoted by $(\bar\fka,\ubar\fkb)\sim(\bar\fka',\ubar\fkb')$ ) if and only if there exists a simplex $\fkb''\subseteq \Gamma_B$ that is complementary to $\bar\fka\cap\bar\fka'$ and so that $\ubar\fkb\cup\ubar\fkb'\subseteq \fkb''$.
 
 \begin{rmk}
  The condition $\ubar\fkb\cup\ubar\fkb'\subseteq \fkb''$ is necessary to ensure that the $B$\=/coordinates of $(\bar\fka,\ubar\fkb)$ and $(\bar\fka',\ubar\fkb')$ are `compatible'. The inclusion will often be strict, as when $i$ is a colour so that $\bar\fka$ and $\bar\fka'$ have differing vertices on the $i$\=/th coordinate, then $\bar\fka\cap\bar\fka'$ will altogether lack vertices on that coordinate.
  Note also that the intersection $\bar\fka\cap\bar\fka'$ is allowed to be empty.
 \end{rmk}

\begin{lem}\label{lem:connect.criterion}
Let $\Gamma_A$ and $\Gamma_B$ be smartly paired. Then $\xcplx$ is connected if and only if $\CG_A$ is connected.
\end{lem}
\begin{proof}
If $(\bar\fka,\ubar\fkb)\sim(\bar\fka',\ubar\fkb')$ is an edge in $\xcplx$, then $(\bar\fka,\ubar\fkb)$ and $(\bar\fka',\ubar\fkb')$ are connected by a path in $\xcplx$ because they are both connected to the vertex $(\bar\fka\cap\bar\fka',\fkb'')\in\xcplx$. 
Now, any vertex $(\fka,\fkb)\in\xcplx$ is connected in $\xcplx$ to $(\bar\fka,\ubar\fkb)$ where $\bar\fka$ is any maximal simplex containing $\fka$ and $\ubar\fkb$ is the subsimplex of $\fkb$ complementary to $\bar\fka$. It follows that if $\CG_A$ is connected then so is $\xcplx$.

Conversely, assume that $\xcplx$ is connected. Then for any two vertices $(\bar\fka,\ubar\fkb)$ and $(\bar\fka',\ubar\fkb')$ in $\CG_A$ there exists a path 
\[
(\bar\fka,\ubar\fkb)=(\fka_0,\fkb_0),\ldots,(\fka_n,\fkb_n)=(\bar\fka',\ubar\fkb')
\]
in the $1$\=/skeleton of $\xcplx$. If $\fka_{i}$ is contained in $\bar\fka'$ for every $0\leq i<n$, then the two vertices actually coincide, as $\bar\fka=\bar\fka'$ by maximality and $\ubar\fkb=\ubar\fkb'$ because at no time any of those $B$\=/coordinates could have been changed.

Otherwise, let $k_1$ be the largest index so that $\fka_{k_1}$ is \emph{not} contained in $\bar\fka'$ and let $(\bar\fka^{(1)},\ubar\fkb^{(1)})\in \CG_A$ be a vertex with $\fka_{k_1}\subseteq\bar\fka^{(1)}$ and $\ubar\fkb^{(1)}\subseteq\fkb_{k_1}$. We claim that $(\bar\fka^{(1)},\ubar\fkb^{(1)})\sim (\bar\fka',\ubar\fkb')$ is an edge in $\CG_A$. 
Indeed, since $(\fka_{k_1},\fkb_{k_1})$ and $(\fka_{k_1+1},\fkb_{k_1+1})$ form an edge, the simplices $\fka_{k_1}$ and $\fka_{k_1+1}$ only differ in one coordinate and are contained in one another. By the choice of $k_1$, $\fka_{k_1+1}$ is contained in $\bar \fka'$, while $\fka_{k_1}$ is not. It is hence $\fka_{k_1}$ that contains $\fka_{k_1+1}$. It follows that $\fka_{k_1+1}\subseteq\bar\fka^{(1)}\cap\bar\fka'$ and that $\ubar\fkb^{(1)}\subseteq\fkb_{k_1}\subsetneq\fkb_{k_1+1}$. Moreover, $\ubar\fkb'\subsetneq \fkb_{k_1+1}$ because $\fka_{k_i+1}\subsetneq\bar\fka'$, therefore $\ubar\fkb'\cup\ubar\fkb^{(1)}\subseteq\fkb_{k_1+1}$. 
It follows that we can pick the unique subsimplex of $\fkb_{k_1+1}$ that is complementary to $\bar\fka^{(1)}\cap\bar\fka'$ to be the simplex $\fkb''$ as in the definition of edge.

To conclude, by gradually enlarging $\fka_{k_1}$ we obtain a path in the $1$\=/skeleton of $\xcplx$ joining $(\fka_{k_1},\fkb_{k_1})$ with $(\bar\fka^{(1)},\ubar\fkb^{(1)})$ so that the $A$\=/coordinates are always subsimplices of $\bar\fka^{(1)}$. Joining the restriction of the path from $(\bar\fka,\ubar\fkb)$ to $(\bar\fka',\ubar\fkb')$ up to $(\fka_{k_1},\fkb_{k_1})$ with this path, we obtain a path connecting $(\bar\fka,\ubar\fkb)$ to $(\bar\fka^{(1)},\ubar\fkb^{(1)})$ so that the largest index $k_2$ for which $\fka_{k_2}$ is not contained in $\bar\fka^{(1)}$ is strictly smaller than $k_1$. The lemma now follows by induction.
\end{proof}

\begin{cor}\label{cor:max dim pure implies connected}
 If $\Gamma_A$, $\Gamma_B$ are smartly paired $n$\=/coloured simplicial complexes of pure dimension $(n-1)$ then $\xcplx$ is connected. 
\end{cor}

\begin{prop}\label{prop:components are CLCC}
 Connected components of CLCCs are CLCCs.
\end{prop}
\begin{proof}
 Let $\xcplx$ be a CLCC. We can assume that $\Gamma_A$ and $\Gamma_B$ are smartly paired. Fix any $(\fka_0,\fkb_0)\in\xcplx$ and choose a $(\bar \fka_0,\ubar\fkb_0)\in\CG_A$ with $\fka_0\subseteq \bar\fka_0$. Let $\CG_A^{(0)}$ be the connected component of $\CG_A$ containing $(\bar\fka_0,\ubar\fkb_0)$. Define:
 \begin{align*}
  \Lambda_A&\coloneqq \bigcup\braces{\bar\fka\mid \exists(\bar\fka,\ubar\fkb)\in\CG_A^{(0)}} \subseteq \Gamma_A \\
  \Lambda_B&\coloneqq \bigcup\braces{\bar\fkb\mid \exists(\bar\fka,\ubar\fkb)\in\CG_A^{(0)},\ \ubar\fkb\subseteq\bar\fkb} \subseteq \Gamma_B.
 \end{align*}

 It follows trivially from the definition that $\Lambda_A$ and $\Lambda_B$ are smartly paired and the connection graph associated with $\xLcplx$ coincides with $\CG_A^{(0)}$. This shows that $\xLcplx$ is connected. 
 
 By functoriality, $\xLcplx$ is naturally identified with a subcomplex of $\xcplx$. The point $(\fka_0,\fkb_0)$ belongs to $\xLcplx$, hence the latter is contained in the connected component of $(\fka_0,\fkb_0)$ in $\xcplx$. We claim that this containment is an equality.
 
 If it was not, there would exist $(\fka,\fkb)\in\xLcplx$ and $(\fka',\fkb')\in\xcplx\smallsetminus\xLcplx$ so that $\{(\fka,\fkb),(\fka',\fkb')\}$ is an edge in $\xcplx$. In particular, $\fka, \fka'$ and $\fkb,\fkb'$ pairwise contain each other. Let $(\bar\fka,\ubar\fkb)\in \CG_A^{(0)}$ and $(\bar\fka',\ubar\fkb')\in \CG_A$ be so that $\fka\subseteq\bar\fka$ and $\fka'\subseteq\bar\fka'$.

 We claim that $(\bar\fka,\ubar\fkb)$ and $(\bar\fka',\ubar\fkb')$ are connected by an edge in $\CG_A$. If $\bar\fka\cap\bar\fka'$ does not have a vertex on the $i$\=/th coordinate, then we see that either $\ubar\fkb$ has one; or $\ubar\fkb'$ has one; or both $\bar\fka$ and $\bar\fka'$ have one, but those vertices differ. In any case, we deduce that $\max\{\fkb,\fkb'\}$ must have a vertex on the $i$\=/coordinate (in the last case we see that $\min\{\fka,\fka'\}$ cannot have a vertex on the $i$\=/th coordinate, else $\bar\fka$ and $\bar\fka'$ would both contain it). This shows that $\max\{\fkb,\fkb'\}$ contains a simplex $\fkb''$ that is complementary to $\bar\fka\cap\bar\fka'$. Since $\ubar\fkb$ and $\ubar\fkb'$ are both contained in $\max\{\fkb,\fkb'\}$, it is easy to see that they are actually contained in $\fkb''$. This proves our claim. It follows that $(\bar\fka',\ubar\fkb')$ belongs to $\CG_A^{(0)}$ as well, and hence $(\fka',\fkb')\in\xLcplx$, yielding the required contradiction.
\end{proof}

\begin{rmk}
 By symmetry, one could of course define a graph $\CG_B$ analogous to $\CG_A$ and use this connection graph to check for connectivity and to relize connected components as coupled\=/links cube complexes. It is not hard to show that the simplicial complexes $\Lambda_A'$ and $\Lambda_B'$ thus obtained coincide with the simplicial complexes constructed using the graph $\CG_A$.
\end{rmk}

\subsection{Hyperplanes}\label{ssec:hyperplanes in CLCC}
 
 Recall that a \emph{hyperplane} $\hpl$ of a cube complex $X$ is an equivalence class of parallel edges of $X$. The {\em carrier} $N(\hpl)$ of a hyperplane $\hpl$, is the union of all cubes having some edge in $\hpl$. 
 
 Hyperplanes can be realized geometrically: if $X=c$ is a single $n$\=/cube, a hyperplane on it can be identified with the $(n-1)$\=/dimensional cube obtained as convex hull of the mid points of the edges in $\hpl$. We call such convex hull the \emph{midcube} associated with $\hpl$.  
 With each hyperplane in a general cube complex $X$ is uniquely associated a cube complex $H$ and a piecewise linear map $H\to X$. Namely, the restriction of $\hpl$ to each cube $c$ in the carrier $N(\hpl)$ determines some (possibly more than one) midcubes of $c$. 
 The cube complex $H$ is constructed from the disjoint union of all such midcubes by gluing the faces that coincide in $X$. The map $H\to X$ is the map restricting to the identity on each midcube (this function is piecewise linear but not cubical because the midcubes are not faces of cubes of $X$). 
 
 We will now show that class of coupled link cube complexes is closed under taking hyperplanes. Namely, let $\xcplx$ be a coupled link cube complex. For a given edge $e\in \xcplx$, let $\hpl(e)$ be the corresponding hyperplane and $H(e)\to X$ the associated piecewise linear map. 
 
 \begin{lem}\label{lem:hyperplanes.are.clcc}
  The cube complex $H(e)$ is a CLCC. Furthermore, the map $H(e)\to \xcplx$ is an embedding.
 \end{lem}

 \begin{proof}
  Let $e$ be an edge on the $i$\=/th coordinate. That is, $e$ is equal to $Q(\fka\ast \{a_i\},\fkb\ast \{b_i\})$ where $a_i, b_i$ are vertices in $\Gamma_A$, $\Gamma_B$ and $\fka\ast \{a_i\}$, $\fkb\ast \{b_i\}$ are simplices of $\Gamma_A$, $\Gamma_B$. Let $\Lambda_A\coloneqq \lk(a_i,\Gamma_A)\ast \{a_i\}$ and $\Lambda_B\coloneqq \lk(b_i,\Gamma_B)\ast\{b_i\}$. The $n$\=/colourings of $\Gamma_A$ and $\Gamma_B$ restrict to $n$\=/colourings of $\Lambda_A$ and $\Lambda_B$. With these colourings, the cube complex $\xLcplx$ is trivially identified to a subcomplex of $\xcplx$ (more formally, $\xLcplx$ is naturally embedded into $\xcplx$ by functoriality). By construction, $e$ is contained in $\xLcplx$. We claim that the carrier $N(\hpl(e))$ is equal to the connected component of $\xLcplx$ containing $e$.
  
  An edge $e'$ parallel to $e$ must be of the form $Q(\fka'\ast \{a_i\},\fkb'\ast \{b_i\})$, where $\fka',\fka$ and $\fkb,\fkb$ differ in one coordinate. The same goes for every edge parallel to $e'$ and, repeating the process, we see that every edge $e''\in\hpl(e)$ is equal to $Q(\fka''\ast \{a_i\},\fkb''\ast \{b_i\})$ for appropriate $\fka''\subseteq\lk(a_i,\Gamma_A),\fkb''\subseteq\lk(b_i,\Gamma_B)$. From this we deduce that the carrier $N(\hpl(e))$ is contained in $\xLcplx$. Since it is connected, $N(\hpl(e))$ is actually contained in a connected component of $\xLcplx$.
  
  It is easy to see that $\xLcplx$ is isomorphic to the direct product 
  \[
   \BX_{\lk(a_i,\Gamma_A),\lk(b_i,\Gamma_B)}\times [0,1],
  \]
  where $\lk(a_i,\Gamma_A)$ and $\lk(b_i,\Gamma_B)$ are seen as $(n-1)$\=/coloured simplicial complexes (see also Lemma~\ref{lem:product CLCC}). In particular, connected components of $\xLcplx$ are direct products of connected components of $\BX_{\lk(a_i,\Gamma_A),\lk(b_i,\Gamma_B)}$ with $[0,1]$. It is immediate to verify that the connected component of $\xLcplx$ containing $e$ is contained in the carrier $N(\hpl(e))$.
  
  By Proposition~\ref{prop:components are CLCC}, connected components of CLCCs are themselves CLCCs. We can therefore find $\Delta_A\subseteq\lk(a_i,\Gamma_A)$ and $\Delta_B\subseteq \lk(b_i,\Gamma_i)$ so that $N(\hpl(e))=\BX_{\Delta_A,\Delta_B}\times[0,1]$.
  As already noted, each edge in $\hpl(e)$ is an edge in the $i$\=/th coordinate. It follows that the map $H\to X$ is an embedding because each cube $c\in N(\hpl(e))$ determines a unique midcube in $H$ (the one orthogonal to the $i$\=/th coordinate). Moreover, we see that $H=\BX_{\Delta_A,\Delta_B}$ because the decomposition $N(\hpl(e))=\BX_{\Delta_A,\Delta_B}\times[0,1]$ shows that each midcube in $H$ corresponds to a cube in $\BX_{\Delta_A,\Delta_B}$.
 \end{proof}
 
 Now that we have an understanding of hyperplanes in CLCCs, we can use this to show that they are special. For more details on special cube complexes see \cite{haglund_special_2008}. 
 
 \begin{de}[Figure~\ref{fig:osculation}]\label{de:special}
  A non-positively curved cube complex $X$ is {\em special} if the hyperplanes of $X$ satisfy the following: 
  \begin{itemize}
  	\item they are embedded, \emph{i.e.} the piecewise linear map $H\to X$ is an embedding; 
  	\item they 2-sided, \emph{i.e.} they separate their carriers in two connected components;
  	\item they do not self-osculate;
  	\item they do not inter-osculate. 
  \end{itemize}
 \end{de}

\begin{figure}
    \centering
    \begin{subfigure}{0.38 \textwidth}
	\centering
	\includegraphics{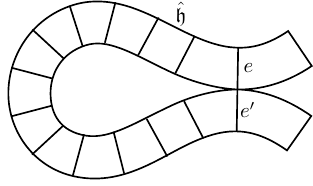}     
    \end{subfigure}\hspace{1 em}
    \begin{subfigure}{0.54 \textwidth}
	\centering
	\includegraphics{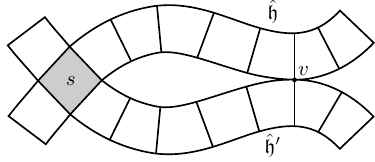}     
    \end{subfigure}
    \caption{Carriers of a self-osculating hyperplane (left) and a pair of inter-osculating ones (right).}
    \label{fig:osculation}
\end{figure}

 \begin{prop}
 	Let $\Gamma_A$ and $\Gamma_B$ be flag, then $\xcplx$ is special.
 \end{prop}
\begin{proof}
	We have seen in Lemma \ref{lem:hyperplanes.are.clcc} that hyperplanes are embedded. 
	It is also clear from the proof of Lemma \ref{lem:hyperplanes.are.clcc} that hyperplanes are 2-sided. 
	Thus we need only check that they do not self- or inter-osculate. 
	
	Let $\Gamma_A$ and $\Gamma_B$ be flag complexes.
	Let $\hpl$ be a hyperplane i.e. an equivalence class of edges of $\xcplx$.
	The hyperplane $\hpl$ self-osculates if there are two edges $e, e'\in \hpl$ which are adjacent to the same vertex of $\xcplx$. 
	However, we saw in the proof of Lemma~\ref{lem:hyperplanes.are.clcc} that there exist $a_i\in\Gamma_A$ and $b_i\in\Gamma_B$ so that each edge of $\hpl$ comes from changing $a_i$ to $b_i$. 
	At each vertex of $\xcplx$ there is at most one such edge $e$. 
	Thus we see that $\hpl$ does not self-osculate. 
	
	We are now left to show that hyperplanes do not inter-osculate. 
	Two hyperplanes $\hpl$ and $\hpl'$ interosculate if they intersect in a square $s$ and there is a vertex $v$ such that both hyperplanes contain edges adjacent to $v$ but do not intersect in any cube adjacent to $v$. 
	Without loss of generality, we will assume that all the edges in the hyperplane $\hpl$ change the first coordinate from $a_1$ to $b_1$ and the edges in $\hpl'$ change the second coordinate from $a_2$ to $b_2$.
	Since $\hpl$ and $\hpl'$ intersect in the square $s$ we see that the first and second coordinates of all the vertices at $s$ can be any of the four choices $(a_1, a_2), (a_1, b_2), (b_1, a_2)$ or $(b_1, b_2)$. 
	Thus we see that $a_1$ and $a_2$ are joined by an edge in $\Gamma_A$. 
	Similarly $b_1$ and $b_2$ are joined by an edge in $\Gamma_B$.
	
	Let $v = (\fka, \fkb)$ be the vertex at which $\hpl$ and $\hpl'$ osculate. 
	Without loss of generality, we can assume that the first two coordinates of $v$ are $(a_1, a_2)$ or $(a_1, b_2)$. 
	Suppose that they are $(a_1, b_2)$. 
	Then the endpoint coming from $\hpl$ has first two coordinates $(b_1, b_2)$ and the endpoint coming from $\hpl'$ has first two coordinates $(a_1,a_2)$.
	In this case both $\fka\ast\{a_2\} = \fka^2$ and $\fkb\ast\{b_1\} = \fkb^1$ are simplices of $\Gamma_A$ and $\Gamma_B$ respectively. 
	Thus we see that $(\fka^2_1, \fkb^1_2)$ is a vertex of $\xcplx$. 
	This show that the edges of $\hpl$ and $\hpl'$ at $v$ are adjacent on a square and thus $\hpl$ and $\hpl'$ intersect in this square. 
	
	Now suppose the first two coordinate of $v$ are $(a_1, a_2)$. 
	Then the endpoint coming from $\hpl$ has first two coordinates $(b_1, a_2)$ and the endpoint coming from $\hpl'$ has first two coordinates $(a_1,b_2)$. 
	Thus we see that $\fkb\ast\{b_1\}$ and $\fkb\ast\{b_2\}$ are both simplices of $\Gamma_B$. 
	We also know that $b_1$ and $b_2$ are connected by an edge. 
	By flagness, we can conclude that $\fkb\ast[b_1, b_2]$ is a simplex of $\Gamma_B$ and thus the vertex $(\fka_{12},\fkb\ast[b_1, b_2])$ is a vertex of $\xcplx$. 
	We now see that the edges defined by $\hpl$ and $\hpl'$ at $v$ are adjacent on a square and thus the hyperplanes intersect in this square. 
\end{proof}

\subsection{Homological dimension}
In this subsection we will introduce some tools to compute the $\ZZ_2$\=/(co)homological dimension of coupled link cube complexes in terms of the homology groups of the defining simplicial complexes.

\begin{de}
 We say that a cube complex $X$ (resp. $\Delta$\=/complex $\Lambda$) of pure dimension $d$ has a \emph{fundamental class} if the $d$\=/chain $\varSigma_X$ (resp. $\varSigma_\Lambda$) given by the sum of all the $d$\=/dimensional cubes (resp. simplices) is a cycle.
\end{de}

Note that a $d$\=/dimensional cube complex (resp. $\Delta$\=/complex) with a fundamental class has non\=/trivial $d$\=/dimensional homology.
 
Given a cube complex $X$ of pure dimension $d$, for every cube $c\subseteq X$ the link $\lk(c,X)$ has pure dimension $d-\dim(c)-1$ and we have $\loc_c(\varSigma_X)=\varSigma_{\lk(c,X)}$ ($\loc_c$ is the localisation map as defined in Section~\ref{sec:prelims}). It follows from Lemma~\ref{lem:n.chains.recognised.locally} that, for every fixed $k<d$, $X$ has a fundamental class if and only if the links of all the $k$\=/cubes in $X$ have a fundamental class. The same argument applies to $\Delta$\=/complexes as well.

For the next lemma, we say that two $n$\=/coloured simplicial complexes are \emph{doubly smartly paired} if they are smartly paired and also every codimension $1$ face of maximal simplices admits a complementary simplex. Note that if $\Gamma_A$ and $\Gamma_B$ are doubly smartly paired then $\xcplx$ has dimension at least $1$.

\begin{lem}\label{lem:fundamental.class.smartly.paired.complexes}
 Let $\Gamma_A$ and $\Gamma_B$ be smartly paired $n$\=/coloured simplicial complexes of pure dimension $k_A$ and $k_B$ respectively. If $\Gamma_A$ and $\Gamma_B$ have a fundamental class then so does $\xcplx$. Moreover, if $\Gamma_A$ and $\Gamma_B$ are doubly smartly paired then the converse holds as well.\oldnote{I am not sure whether this converse holds if I only assume smartly paired}
\end{lem}
\begin{proof}
 By Corollary~\ref{cor:dimension.xcplx.smartly.paired}, $\xcplx$ is a cube complex of pure dimension $d=k_A+k_B+2-n$. As above, let $\varSigma_\xcplx$, $\varSigma_{\Gamma_A}$ and $\varSigma_{\Gamma_B}$ be the chains obtained by summing all the top dimensional cubes (simplices). We have to show that if $\varSigma_{\Gamma_A}$ and $\varSigma_{\Gamma_B}$ are cycles then so is $\varSigma_{\xcplx}$.
 
 For every vertex $v\coloneqq(\fka,\fkb)$ in $\xcplx$, by Lemma~\ref{lem:links.in.xcplx} we have 
\[
\lk\bigparen{v,\xcplx}=\lk(\fka,\Gamma_A)*\lk(\fkb,\Gamma_B),
\]
 whence we deduce
 \begin{align*}
  \loc_v\bigparen{\varSigma_\xcplx} &=\varSigma_{\lk\paren{v,\xcplx}} 	\vphantom{\big(} \\
    &=\varSigma_{\paren{ \lk(\fka,\Gamma_A)*\lk(\fkb,\Gamma_B)} }  	\vphantom{\big(}\\
    &=\varSigma_{\lk(\fka,\Gamma_A)}*\varSigma_{\lk(\fkb,\Gamma_B)} 	\vphantom{\big(}\\
    &=\loc_\fka\paren{\varSigma_{\Gamma_A} }*\loc_\fkb\paren{\varSigma_{\Gamma_B} }.
 \end{align*}
 It follows from Lemma~\ref{lem:n.chains.recognised.locally} and Lemma~\ref{lem:joinsofcycles} that if $\varSigma_{\Gamma_A}$ and $\varSigma_{\Gamma_B}$ are cycles then $\varSigma_\xcplx$ is a cycle as well.
 
 Conversely, assume now that $\Gamma_A$ and $\Gamma_B$ are doubly smartly paired and that $\varSigma_{\xcplx}$ is a cycle. For every simplex $\fka\subsetneq\Gamma_A$ of dimension $k_A-1$ there exists a $\fkb\subseteq\Gamma_B$ so that $v\coloneqq(\fka,\fkb)$ is a vertex of $\xcplx$. 
 From the discussion above, it follows that $\loc_\fka\paren{\varSigma_{\Gamma_A} }*\loc_\fkb\paren{\varSigma_{\Gamma_B} }$
 is a cycle.
 
 Neither $\loc_\fka\paren{\varSigma_{\Gamma_A} }$ nor $\loc_\fkb\paren{\varSigma_{\Gamma_B} }$ is a trivial chain (even if $\lk(\fkb,\Gamma_B)$ was empty we would still get a non\=/trivial $(-1)$\=/chain). \oldnote{I should check that everything we do is coherent when we have empty complexes and $-1$-chains lurking around}
 It hence follows from Lemma~\ref{lem:joinsofcycles} that $\loc_\fka\paren{\varSigma_{\Gamma_A} }$ is a cycle. Since this is the case for every $(k_A-1)$\=/simplex $\fka$ in $\Gamma_A$, we deduce that $\varSigma_A$ is a cycle by Lemma~\ref{lem:n.chains.recognised.locally}. The same argument shows that $\varSigma_B$ is a cycle as well.
\end{proof}

Together with Corollary~\ref{cor:dimension.xcplx.smartly.paired}, we obtain the following: 

\begin{cor}
 Let $\Gamma_A$ and $\Gamma_B$ be smartly paired $n$\=/coloured simplicial complexes of pure dimension $k_A$ and $k_B$ respectively. If $\Gamma_A$ and $\Gamma_B$ have a fundamental class then $\pi_1(\xcplx)$ has homological dimension $k_A+k_B+2-n$. 
\end{cor}

To prove a more refined result we need to give another definition. The \emph{support} of a $k$\=/chain $\varOmega$ in a $\Delta$\=/complex is the subcomplex given by the union of the $k$\=/simplices appearing with non\=/trivial coefficient in $\varOmega$.

\begin{de}\label{def:smartlypairedchains}
 Let $\Gamma_A$ and $\Gamma_B$ be $n$\=/coloured simplicial complexes. Given a $d_A$\=/chain $\varOmega_A\in H_{d_A}\paren{\Gamma_A;\ZZ_2}$ and a $d_B$\=/chain $\varOmega_B\in H_{d_B}\paren{\Gamma_B;\ZZ_2}$, we say that $\varOmega_A$ and $\varOmega_B$ are \emph{smartly paired chains} if every $d_A$\=/simplex in the support of $\varOmega_A$ has a complementary simplex in the support of $\varOmega_B$ and, vice versa, every $d_B$\=/simplex in the support of $\varOmega_B$ has a complementary simplex in the support of $\varOmega_A$.
\end{de}

\begin{thm}\label{thm:chains.to.chains,cycles.to.cycles,amen}
 If $\varOmega_A$ and $\varOmega_B$ are smartly paired $d_A$- and $d_B$\=/chains in two $n$\=/coloured simplicial complexes $\Gamma_A$ and $\Gamma_B$, then they define a $(d_A+d_B+2-n)$\=/chain $\BX(\varOmega_A,\varOmega_B)$ in $\xcplx$. If $\varOmega_A$ and $\varOmega_B$ are cycles also $\BX(\varOmega_A,\varOmega_B)$ is a cycle. 
 
 Moreover, if $\BX(\varOmega_A,\varOmega_B)$ is a boundary, then for every vertex $(\fka,\fkb)$in $\xcplx$ the chain $\loc_\fka\paren{\varOmega_A}\ast\loc_\fkb\paren{\varOmega_B}$ must be a boundary in $\lk(\fka,\Gamma_A)\ast\lk(\fkb,\Gamma_B)$.\oldnote{it is quite possible that at least one between $\varOmega_A$ and $\varOmega_B$ must be a boundary itself, but I don't quite know how to prove it}
\end{thm}

\begin{proof}
 Let $\chi_A\subseteq\Gamma_A$ and $\chi_B\subseteq\Gamma_B$ be the supports of $\varOmega_A$ and $\varOmega_B$ respectively. By the definition we have that $\chi_A$ is a simplicial complex of pure dimension $d_A$ and that $\varOmega_A$ coincides with the $d_A$\=/chain $\varSigma_{\chi_A}$ given by the sum of all the top dimensional simplices. The same goes for $\chi_B$ as well, and we also have that the two $n$\=/coloured complexes $\chi_A$ and $\chi_B$ are smartly paired because so are $\varOmega_A$ and $\varOmega_B$. It follows from Corollary~\ref{cor:dimension.xcplx.smartly.paired} that $\BX_{\chi_A,\chi_B}$ is a cube complex of pure dimension $d_A+d_B+2-n$ and hence the sum of its top dimensional cubes gives us a $(d_A+d_B+2-n)$\=/chain $\Sigma_{\BX_{\chi_A,\chi_B}}$ in $\BX_{\chi_A,\chi_B}$. By functoriality, we have a natural inclusion of $\BX_{\chi_A,\chi_B}$ into $\xcplx$. We define $\BX(\varOmega_A,\varOmega_B)$ to be the image of $\Sigma_{\BX_{\chi_A,\chi_B}}$ under this inclusion.
 
 Since the inclusion is an injective cubical map, $\BX\paren{\varOmega_A,\varOmega_B}$ is a cycle if and only if $\varSigma_{\BX_{\chi_A,\chi_B}}$ is. By Lemma~\ref{lem:fundamental.class.smartly.paired.complexes} if $\varSigma_{\chi_A}$ and $\varSigma_{\chi_B}$ (equivalently, $\varOmega_A$ and $\varOmega_B$) are cycles, then this is indeed the case. This concludes the first part of the proof.
 
 For the second part of the statement, unravelling the definition it turns out that for every vertex $v=(\fka,\fkb)\in\xcplx$ the localisation at $v$ of $\BX\paren{\varOmega_A,\varOmega_B}$ is given by
 \[
  \loc_v\bigparen{\BX\paren{\varOmega_A,\varOmega_B}}=\loc_\fka\paren{\varOmega_A}\ast \loc_\fkb\paren{\varOmega_B}.
 \]
 Then the statement follows from Lemma~\ref{lem:n.chains.recognised.locally}.
\end{proof}

\begin{rmk}\label{rmk:question on homology}
 It would be interesting to see to what extent the results of this subsection generalise to homology with coefficients other than $\ZZ_2$. This would be very useful to study topological properties of CLCCs, \emph{e.g.} to prove orientability.
 
 The reason why it is fairly simple to study the homology with $\ZZ_2$\=/coefficients is that the cellular chain complex of a cube complex behaves like the simplicial chain complex of a simplicial complex.
 This suggest that it may be possible to study homology groups with other coefficients by first developing a theory of simplicial cubical homology.
\end{rmk}

\section{Examples}
\label{sec:Examples}

In this section we provide some concrete examples of cube complexes obtainable as coupled link cube complexes.

\subsection{Surfaces}\label{ssec:examples.surfaces}
For any two integers $k_A>1$, $k_B>1$, let $\Gamma_A$ and $\Gamma_B$ be the graphs consisting of a cycle of length $2k_A$ and $2k_B$ respectively. Choosing a $2$\=/colouring on each cycle, we can construct a coupled link square complex $\xcplx$ and we can use Lemma \ref{lem:links.in.xcplx} to compute the links of the vertices of $\xcplx$. A vertex $v\in\xcplx$ can be of three types:
\begin{enumerate}[(1)]
 \item $v=(\fka,\emptyset)$ and therefore $\lk(v,\xcplx)=\Gamma_B$;
 \item $v=(\fka,\fkb)$ and therefore $\lk(v,\xcplx)=\SS^0\ast \SS^0$;
 \item $v=(\emptyset,\fkb)$ and therefore $\lk(v,\xcplx)=\Gamma_A$.
\end{enumerate}
In either case the link is a circle, therefore $\xcplx$ is the cubulation of a surface (see also Fact~\ref{fact:cc.with.spheres.as.links.are.mflds} below). Since both $\Gamma_A$ and $\Gamma_B$ have simplices with a vertex in every coordinate, it follows from Corollary~\ref{cor:max dim pure implies connected} that the surface thus obtained is connected.

It is not very hard to verify that these surfaces are orientable. One way to see this is as follows. Numerate the vertices of $\Gamma_A,\Gamma_B$, then each edge in these cycles will connect an even number $2i$ with $2i+1$ or $2i-1$. In the first case we say that the edge is \emph{ascending}, in the latter it is \emph{descending}.  Each square in $\xcplx$ is equal $Q(\fka,\fkb)$ for appropriate edges $\fka$, $\fkb$ in $\Gamma_A$, $\Gamma_B$. Orient such a square clockwise if $\fka$ and $\fkb$ are of the same type (both ascending or both descending), anti\=/clockwise otherwise. These local orientations are coherent and define a global orientation on $\xcplx$.

Since every edge of $\xcplx$ is contained in two squares and every square has four vertices, we deduce that the Euler characteristic of $\xcplx$ is 
\[
 \chi\bigparen{\xcplx}=\frac{1}{4}\smashoperator[r]{\sum_{v\in\xcplx}}\bigparen{4- \deg(v)}
\]
where $\deg(v)$ is the number of edges containing the vertex $v$.

It follows that $\chi\bigparen{\xcplx}=-k_A(k_B-2)-k_B(k_A-2)$ and therefore for every $g\geq 1$ the surface $\Sigma_g$ of genus $g$ can be obtained as a coupled link square complex  \emph{e.g.} by letting $k_A=2$ and $k_B=g+1$.

\subsection{PL Manifolds}\label{ssec:examples.manifolds}
CLCCs can be used to construct piecewise\=/linear manifolds (PL manifolds). Recall that a {\em PL triangulation} of an $n$\=/dimensional manifold $M$ is a triangulation so that the link each vertex is a PL sphere of one dimension $n-1$ (this is an inductive definition). A {\em PL manifold} is a manifold that can be equipped with a PL triangulation. 
In the following, we use the term triangulations of PL manifolds to mean PL triangulations.
We will make use of the following facts to prove that some CLCCs are (cubulated) PL manifolds:

\begin{fact}\label{fact:cc.with.spheres.as.links.are.mflds}
  Let $X$ be an $n$\=/dimensional finite cube complex. If $\lk(v,X)$ is PL homeomorphic to $\SS^{n-1}$ for every vertex $v\in X^{(0)}$, then $X$ is a closed PL manifold.
\end{fact}
 
\begin{fact}\label{fact:links.in.mflds.are.spheres}
  If $X$ is a triangulated $n$\=/dimensional closed PL manifold, then for every $k$\=/dimensional simplex $\sigma\subsetneq X$ with $k<n$ the link $\lk(\sigma,X)$ is a triangulated sphere of dimension $n-k-1$.
\end{fact}

Both facts are fairly elementary. For a proof, see resp. Lemma 9 and the subsequent Corollary 1 in \cite[Chapter 3]{zeeman1966seminar}.\footnote{%
In \cite[Chapter 3]{zeeman1966seminar} the word ``manifold'' is short for polyhedral manifold. This concept is slightly more general than PL manifold. However, the same proof holds for PL manifolds as well.
}

\

If two $n$\=/coloured simplicial complexes $\Gamma_A$ and $\Gamma_B$ are triangulations of PL spheres, it follows from Fact~\ref{fact:links.in.mflds.are.spheres} and Lemma~\ref{lem:links.in.xcplx} that the link of every vertex in $\xcplx$ is a join of spheres and it is hence itself a sphere. It then follows from Fact~\ref{fact:cc.with.spheres.as.links.are.mflds} that $\xcplx$ is a closed PL manifold (when it is not empty). 

Thus, we obtain the following: 
\begin{prop}\label{prop:manifoldsifspheres}
	Let $\Gamma_A$ and $\Gamma_B$ be $n$\=/coloured PL triangulations of spheres. Then $\xcplx$ is a PL manifold. 
\end{prop}

If $\Gamma_B$ is a sphere of dimension strictly less than $n-1$ then the same argument implies that $\xcplx$ is a closed PL manifold also when $\Gamma_A$ is a PL triangulation of any PL manifold (not necessarily a sphere). Indeed, every vertex $v$ in $\xcplx$ will have at least one $A$\=/coordinate and hence the contribution to the $A$\=/coordinates of the link at $v$ will be given by the link of a non\=/empty simplex of $\Gamma_A$, which is again a sphere. Similarly, we do not even need to assume $\Gamma_B$ to be a sphere as long as $\Gamma_A$ does not have dimension $n-1$.

For a more concrete example, we can let $\Gamma_A=\Ast_{i=1}^n \SS^0$ with the standard $n$\=/colouring assigning to the $i$\=/th $\SS^0$ the colour $i$. Then, letting $\Gamma_B$ be any $n$\=/coloured sphere, the resulting CLCC $\xcplx$ will be a non\=/empty PL manifold of dimension $\dim(\xcplx)=\dim(\Gamma_B)+1$. Note that $\dim(\xcplx)$ does not depend on the colouring of $\Gamma_B$, however, its homeomorphism class does. For example, let $\Gamma_A=\SS^0\ast\SS^0\ast\SS^0$ and let $\Gamma_B$ be a $3$\=/coloured circle of length $6$. If the chosen $3$\=/colouring cycles through the three colours it turns out that $\xcplx$ is the surface of genus $3$. Conversely, if we choose the $3$\=/colouring that is only alternating two colours (\emph{i.e.} we completely ignore one of the colours), we will get a disjoint union of two surfaces of genus $2$. 

We do not kow the answer to the following questions.

\begin{question}
 If a CLCC is a PL manifold is it always orientable? 
\end{question}

We suspect that the answer should be positive. One way to approach to this question could be to develop the study of homology groups of CLCC with $\ZZ$ coefficients (as outlined in Remark~\ref{rmk:question on homology}). Both problems seem to share a certain level of difficulty in coherently choosing local orientations.

\begin{question}
 If a CLCC is a PL manifold is it always smoothable? (That is, does there need to be a compatible differential structure?) If not, is it possible to characterise when this happens?  
\end{question}

\subsection{Pseudo\=/manifolds}
\label{ssec:examples.pseudomanifolds}
A \emph{$k$\=/dimensional pseudo\=/manifold} is a simplicial complex of pure dimension $k$ such that
\begin{itemize}
 \item every $(k-1)$\=/simplex belongs to exactly two $k$\=/simplices (\emph{non\=/branching}),
 \item for every pair of $k$\=/simplices $\sigma,\sigma'$ there is a sequence of $k$\=/simplices $\sigma=\sigma_1,\ldots,\sigma_l=\sigma'$ such that $\sigma_i$ and $\sigma_{i+1}$ share a $(k-1)$\=/dimensional simplex (\emph{strongly connected}).
\end{itemize}

We will soon see that coupled link cube complexes can be used to construct pseudo\=/manifolds. In order to do so, it is necessary to subdivide cubes into simplices in order to realise the cube complex as a simplicial complex. One convenient way for doing this is to take the \emph{simplicial barycentric subdivision}. Namely, the cube complex $X$ is made into a simplicial complex $\Lambda$ whose vertex set is equal to the set of cubes of $X$ (each vertex correpond to the barycentre of the associated cube), and whose simplices correspond to chains of nested cubes of increasing dimension. This way, each edge is split into two edges, each square is split into eight triangles and so on. In order for $\Lambda$ to be a pseudo\=/manifold it is enough that $X$ satisfies the cubical analogues of non-branching and strong connectedness.

\begin{lem}
 Let $\Gamma_A$ and $\Gamma_B$ be smartly paired $n$\=/coloured $(n-1)$\=/dimensional pseudo\=/manifolds. Then the simplicial barycentric subdivision of $\xcplx$ is a pseudo\=/manifold.
\end{lem}
\begin{proof}
 The cube complex $\xcplx$ has pure dimension $n$. We need to show that it satisfies the cubical analogues of non-branching and strong connectedness.

 Let $Q(\fka,\fkb)$ be a $(n-1)$\=/cube. We can assume that $\fka$ is a $(n-2)$\=/simplex and $\fkb$ is a $(n-1)$\=/simplex. Since $\Gamma_A$ is a pseudo\=/manifold, $\fka$ is contained in exactly two $(n-1)$\=/simplices $\fka'$, $\fka''$. In turn, $Q(\fka,\fkb)$ is only contained in $Q(\fka',\fkb)$ and $Q(\fka'',\fkb)$. This proves that $\xcplx$ is non\=/branching.
 
 Recall that $n$\=/dimensional cubes of $\xcplx$ correpond to pairs of $(n-1)$\=/dimensional simplices of $\Gamma_A$, $\Gamma_B$.
 Let $Q(\fka,\fkb)$ and $Q(\fka',\fkb)$ be two $n$\=/cubes such that $\fka\cap\fka'$ is $(n-2)$\=/dimensional. Note that $Q(\fka,\fkb)\cap Q(\fka',\fkb)= Q(\fka\cap\fka',\fkb)$. The assumption that $\Gamma_B$ is $n$\=/coloured implies that $Q(\fka\cap\fka',\fkb)$ is a non\=/empty $(n-1)$\=/cube.
 
 Since $\Gamma_A$ is strongly connected, it follows that any two $n$\=/cubes of the form $Q(\fka,\fkb)$ and $Q(\fka',\fkb)$ are connected by an appropriate sequence in $\xcplx$. The same goes for pairs of $n$\=/cubes of the form $Q(\fka,\fkb)$ and $Q(\fka,\fkb')$, because $\Gamma_B$ is strongly connected. It follows that $\xcplx$ is strongly connected.
\end{proof}

\begin{rmk}
 The assumption that the $(n-1)$\=/dimensional pseudo\=/manifolds be $n$\=/coloured was used to ensure that $\xcplx$ is strongly connected. One could relax this assumption, \emph{e.g.} by defining an appropriate notion of pairwise strongly connected simplicial complexes.
\end{rmk}

\subsection{Products of CLCCs}
It is easy to observe that products of coupled link cube complexes are themselves coupled link cube complexes.

More precisely, let $\Gamma_A^{1}$, $\Gamma_B^{1}$ be $n$\=/coloured with colours $1,\ldots,n$, and let $\Gamma_A^{2}$, $\Gamma_B^{2}$ be $m$\=/coloured with colours $n+1,\ldots,n+m$. This way the joins $\Gamma_A^{1}\ast\Gamma_A^{2}$, $\Gamma_B^{1}\ast\Gamma_B^{2}$ are $n+m$ coloured simplicial complexes and we have the following.

\begin{lem}\label{lem:product CLCC}
 The CLCC $\BX_{\Gamma_A^{1}\ast\Gamma_A^{2},\Gamma_B^{1}\ast\Gamma_B^{2}}$ is equal to the product $\BX_{\Gamma_A^{1},\Gamma_B^{1}}\times \BX_{\Gamma_A^{2}, \Gamma_B^{2}}$.
\end{lem}
\begin{proof}
 By the definition, a simplex in $\Gamma_A^{1}\ast\Gamma_A^{2}$ is a join of a simplex in $\Gamma_A^{1}$ with a simplex in $\Gamma_A^{2}$. The same goes for $\Gamma_B^{1}\ast\Gamma_B^{2}$. Two simplices $\fka^1\ast\fka^2\subseteq\Gamma_A^1\ast\Gamma_A^2$ and $\fkb^1\ast\fkb^2\subseteq\Gamma_B^1\ast\Gamma_B^2$ are complementary for $\Gamma_A^{1}\ast\Gamma_A^{2},\Gamma_B^{1}\ast\Gamma_B^{2}$ if and only if $\fka^1$, $\fkb^1$ are complementary for $\Gamma_A^1$, $\Gamma_B^1$ and $\fka^2$, $\fkb^2$ are complementary for $\Gamma_A^2$, $\Gamma_B^2$. This shows that there is a one\=/to\=/one correspondence between the vertices of $\BX_{\Gamma_A^{1}\ast\Gamma_A^{2},\Gamma_B^{1}\ast\Gamma_B^{2}}$ and those of the product $\BX_{\Gamma_A^{1},\Gamma_B^{1}}\times \BX_{\Gamma_A^{2}, \Gamma_B^{2}}$.

 Edges are defined by flipping one coordinate at the time. In particular, every edge in $\BX_{\Gamma_A^{1}\ast\Gamma_A^{2},\Gamma_B^{1}\ast\Gamma_B^{2}}$ determines an edge in $\BX_{\Gamma_A^{1},\Gamma_B^{1}}$ or $\BX_{\Gamma_A^{2}, \Gamma_B^{2}}$ depending on whether the coordinate being flipped is in $\{1,\ldots,n\}$ or $\{n+1,\ldots,n+m\}$. Vice versa with any vertex in $\BX_{\Gamma_A^{1}\ast\Gamma_A^{2},\Gamma_B^{1}\ast\Gamma_B^{2}}$ and edge in $\BX_{\Gamma_A^{1},\Gamma_B^{1}}$ or $\BX_{\Gamma_A^{2}, \Gamma_B^{2}}$ is uniquely associated an edge in $\BX_{\Gamma_A^{1}\ast\Gamma_A^{2},\Gamma_B^{1}\ast\Gamma_B^{2}}$. This shows that the correspondence between vertices of $\BX_{\Gamma_A^{1}\ast\Gamma_A^{2},\Gamma_B^{1}\ast\Gamma_B^{2}}$ and $\BX_{\Gamma_A^{1},\Gamma_B^{1}}\times \BX_{\Gamma_A^{2}, \Gamma_B^{2}}$ preserves the edges.
 
 Both cube complexes are determined by their $1$\=/skeleta, as the higher dimensional cubes are just filled in whenever possible. This implies that they are indeed isomorphic. 
\end{proof}

\subsection{Right Angled Artin Groups}\label{ssec:examples.RAAGs}
Let $\Gamma$ be a flag simplicial complex with vertices $v_1,\ldots,v_n$. Recall that the \emph{right angled Artin group} associated with $\Gamma$ is the group 
\[
\CA(\Gamma)\coloneqq \angles{v_1,\ldots, v_n\mid v_iv_j=v_jv_i \text{ if and only if }(v_i,v_j)\text{ is an edge of }\Gamma}.
\]

The \emph{Salvetti complex} of $\Gamma$ can be defined as follows. Let $\TT^n=\SS^1\times\cdots\times\SS^1$ be the $n$ dimensional torus endowed with the natural cell complex structure with a single vertex and ${n \choose k}$ $k$\=/cells. There is a bijection between the 1-cells and the vertices of $\Gamma$. The Salvetti complex\footnote{In the literature, the term `Salvetti complex' is often used to denote the universal cover of the complex ${\rm Sal}(\Gamma)$ here defined.} is defined as the subcomplex ${\rm Sal}(\Gamma)\subseteq \TT^n$ that contains a $k$-cell if and only if the vertices $v_{i_1},\ldots,v_{i_k}$ associated with the $1$\=/cells in its boundary span a simplex of $\Gamma$. It is well known that $\CA(\Gamma)$ is isomorphic to the fundamental group of ${\rm Sal}(\Gamma)$.

For $i=1,\ldots,n$ let $A_i$ and $B_i$ be sets with two elements ($a^+,a^-$ and $b^+,b^-$ respectively). With the notation of Section \ref{sec:complex_X_Gamma}, letting $\BA=\Ast_{i=1}^n A_i$ and $\BB=\Ast_{i=1}^n B_i$ yields the complex $\BK=\prod_{i=1}^n A_i\ast B_i$ which is equal to the cube complex obtained from $\TT^n$ by taking the cubical barycentric subdivision twice. In particular, (a subdivision of) the Salvetti complex is naturally a subcomplex of $\BK$.

Let $\widehat \Gamma$ be the subcomplex of $\BA$ obtained as follow. For each (possibly empty) simplex $[v_{i_1},\ldots,v_{i_k}]\subseteq \Gamma$, add to $\widehat \Gamma$ the $n$\=/dimensional simplex with $a^+$ on the coordinates $i_1,\ldots,i_k$ and $a^-$ on the remaining $n-k$ coordinates. Since $\Gamma$ is flag, $\widehat \Gamma$ is also a flag simplicial complex.

\begin{figure}
    \centering
    \begin{subfigure}{0.2 \textwidth}
	\centering
	\includegraphics{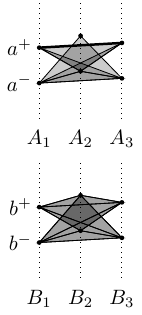}     
    \end{subfigure}\hspace{1 em}
    \begin{subfigure}{0.75 \textwidth}
	\centering
	\includegraphics{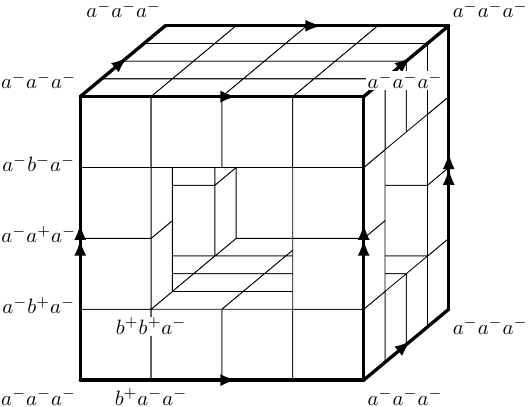}     
    \end{subfigure}
    \caption{The CLCC associated with $\CA(\Gamma)$ where $\Gamma$ is the graph with $3$ vertices and a single edge.}
    \label{fig:Salvetti_cplx_coordinates}
\end{figure}

\begin{lem}
The complex $\BX_{\widehat\Gamma,\BB}\subseteq \BK$ contains the Salvetti complex ${\rm Sal}(\Gamma)\subseteq\BK$ and it deformation retracts onto it.
\end{lem}

\begin{proof}
Identify $\TT^n$ with the quotient of $[0,1]^n$ obtained gluing opposite faces. For every subset $I\subseteq\{1,\ldots,n\}$ let 
\[
C_I\coloneqq \{(x_1,\ldots,x_n) \mid x_i=1/2 \text{ for all }i\in I\}\subseteq \TT^n
\] 
be the fibre over the middle point of the projection of $\TT^n$ onto the $I$ coordinates. Let $Y\subsetneq\TT^n$ be the subset obtained removing $C_I$ from $\TT^n$ every time the set $\{v_i \mid i\in I\}$ does \emph{not} span a simplex of $\Gamma$. Then one can use radial projections to show that $Y$ deformation retracts onto ${\rm Sal}(\Gamma)$. More precisely, if $\Gamma$ contains the full simplex $\{v_1,\ldots,v_n\}$ there is nothing to prove. In the other case, consider the deformation retract of $\TT^n\smallsetminus \{(1/2,\ldots,1/2)\}$ onto its $(n-1)$\=/skeleton by radial projection: every set of $(n-1)$ coordinates identifies an $(n-1)$\=/dimensional torus whose centre is in the image of $Y$ if and only if it also belongs to the Salvetti complex. When the centre is missing perform another radial projection and iterate the process. 

Now, identify $\BK$ with $\TT^n$ by letting $a^-=0,\ b^-=1/4,\ a^+=1/2,\ b^+=3/4$ on every coordinate. Then a vertex $v\in\BK$ is \emph{not} in $\BX_{\widehat\Gamma,\BB}\subseteq \BK$ if and only if the set of indices whose coordinates equal $a^+$ does \emph{not} span a simplex in $\Gamma$. That is, $v\notin \BX_{\widehat\Gamma,\BB}$ if and only if $v\notin Y$ (see Figure~\ref{fig:Salvetti_cplx_coordinates}). It follows easily that the deformation retract of $Y$ onto ${\rm Sal}(\Gamma)$ restricts to a deformation retract of $\BX_{\widehat\Gamma,\BB}$.   
\end{proof}

\begin{cor}
Every right angled Artin group can be obtained as the fundamental group of a CLCC.
\end{cor}

\subsection{(Commutators of) Right Angled Coxeter Groups}\label{ssec:examples.RACGS}
Recall now that the \emph{right angled Coxeter group} associated with $\Gamma$ is the quotient 
\[
\CC(\Gamma)\coloneqq\CA(\Gamma)/\aangles{{v_i^2 \mid i=1,\ldots ,n}}.
\] 
Its commutator subgroup $[\CC(\Gamma),\CC(\Gamma)]$ has finite index in $\CC(\Gamma)$ and it is proved in \cite{Dro03} that it is the fundamental group of the cube complex $K(\Gamma)\subseteq [0,1]^n$ containing a $k$\=/dimensional cube $c$ if and only if the indices of the $k$ coordinates that vary within $c$ span a simplex of $\Gamma$. The link of each vertex of $K(\Gamma)$ is equal to $\Gamma$, hence $K(\Gamma)$ is non\=/positively curved.

Let again $\BA\coloneqq \Ast_{i=1}^n\{a^-,a^+\}$ but this time let $\BB=\Ast_{i=1}^n\{b\}$ be a single simplex. Then $\BK$ is equal to the cubical barycentric subdivision of $[0,1]^n$ and therefore $K(\Gamma)$ can be identified with a subset of $\BK$. Note that $\Gamma$ is naturally a subcomplex of $\BB$. Then it is simple to prove the following:

\begin{lem}\label{lem:commutators.RACG.is.CLCC}
The complex $\BX_{\BA,\Gamma}$ is equal to the barycentric subdivision of $K(\Gamma)$.
\end{lem}

\begin{proof}
It is enough to show that $\BX_{\BA,\Gamma}$ and $K(\Gamma)$ coincide as subsets of $\BK$. Note that (the cubical barycentric subdivision of) a $k$\=/cube $c\subseteq [0,1]^n$ is contained in $\BX_{\BA,\Gamma}$ if and only if $\BX_{\BA,\Gamma}$ contains the barycenter of $c$. The lemma follows because the set of $B$\=/coordinates of the barycenter of $c$ coincides with the set of coordinates that vary in $c$ (see Figure~\ref{fig:Coxeter_cplx_coordinates}).
\end{proof}

\begin{figure}
    \centering
    \begin{subfigure}{0.3 \textwidth}
	\centering
	\includegraphics{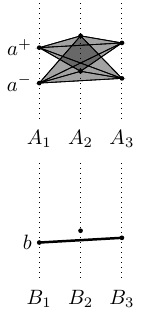}     
    \end{subfigure}%\hspace{2 em}
    \begin{subfigure}{0.5 \textwidth}
	\centering
	\includegraphics{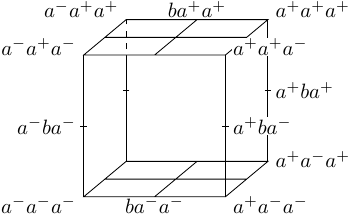}     
    \end{subfigure}
    \caption{The CLCC associated with the commutators of $\CC(\Gamma)$ where $\Gamma$ is the graph with $3$ vertices and a single edge.}
    \label{fig:Coxeter_cplx_coordinates}
\end{figure}

\begin{cor}
The commutator subgroup of every right angled Coxeter group can be realised as the fundamental group of a CLCC.
\end{cor}

\section{Hyperbolicity criteria for coupled link cube complexes}\label{sec:hyperbolic}

\subsection{Hyperplanes in cube complexes}
In this section we recall some basic facts about hyperplanes in CAT(0) cube complexes. 
 Recall from Subsection \ref{ssec:hyperplanes in CLCC} that a hyperplane $\hpl$ of $X$ is an equivalence class of parallel edges of $X$. If $X$ is CAT(0), the associated map $H\to X$ is an embedding whose image is the CAT(0) cube complex spanned by the midpoints of the edges in $\hpl$ (\emph{i.e.} the smallest convex set containing them). We can---and will---identify a hyperplane $\hpl$ with this subset. Moreover, every hyperplane $\hpl$ disconnects $X$ into two \emph{halfspaces} $\hsp$ and $\hsp^*$. 
After taking a cubical subdivision we have that hyperplanes, their carriers and halfspaces are all convex subcomplexes of $X$. As such they satisfy the Helly property: 

\begin{theorem}[{\cite[Theorem 2.2]{roller}}]
	Let $C_1, \dots, C_n$ be convex subcomplexes of a CAT(0) cube complex $X$. Suppose that $C_i\cap C_j\neq\emptyset$ for all $i, j$. 
	Then $C_1\cap\dots\cap C_n\neq \emptyset$. 
\end{theorem}

The following consequence will be useful in what follows. 

\begin{cor}\label{cor:carriers}
	Let $\hpl_1, \dots, \hpl_n$ be hyperplanes. Suppose that $N(\hpl_i)\cap N(\hpl_j)\neq \emptyset$ for all $i, j$. Then $N(\hpl_1)\cap \dots\cap N(\hpl_n)$ contains a vertex of $X$. 
\end{cor}

\begin{de}	
	We say that two distinct hyperplanes $\hpl, \hpl'$ are {\em transverse} if they intersect. We say that a hyperplane $\hpl$ {\em separates} two subsets $Y, Z$ of $X$, if $Y \subseteq \hsp$ and $Z\subseteq \hsp^*$.
\end{de}

\begin{lem}\label{lem:transverse hyperplane intersect}
 If $\hpl_1$ and $\hpl_2$ are transverse, then $N(\hpl_1)\cap N(\hpl_2)$ is the union of all the cubes intersecting $\hpl_1\cap\hpl_2$.
\end{lem}
\begin{proof}
 Let $N(\hpl_1\cap\hpl_2)$ be the union of the cubes interesting $\hpl_1\cap\hpl_2$. It is clear that $\emptyset\neq N(\hpl_1\cap\hpl_2)\subseteq N(\hpl_1)\cap N(\hpl_2)$. 
 Since $N(\hpl_1)\cap N(\hpl_2)$ is connected, if the other containment did not hold there would be a vertex $w\in N(\hpl_1\cap \hpl_2)$ and an edge $(w,w')\in N(\hpl_1)\cap N(\hpl_2)$ with $w'\notin N(\hpl_1\cap\hpl_2)$. This edge $(w,w')$ can be used to construct an empty triangle in the link of $w$, against flag condition.
\end{proof}

\begin{de}
	A {\em grid of hyperplanes} is the data of two families of hyperplanes $V=\{\hplv_1,\dots,\hplv_p\}$ and $H=\{\hpl_1,\dots,\hpl_q\}$, 
	such that any $\hplv_i$ is transverse to any $\hpl_j$, and any $\hplv_i$ (resp. $\hpl_j$) separates $\hplv_{i-1}$ and $\hplv_{i+1}$ (resp. $\hpl_{j-1}$ and $\hpl_{j+1}$).
	Such a grid is {\em $\delta$-thin} if $\min(p, q)\leq \delta$. 
	We say that a grid is {\em maximal} if for each $i$ there is no hyperplane separating $\hpl_i, \hpl_{i+1}$ and no hyperplane separating $\hplv_i, \hplv_{i+1}$. 
	
	We will refer to the hyperplanes in $V$ as {\em vertical} and the hyperplanes in $H$ as {\em horizontal}.
\end{de}

\begin{lem}\label{lem:maxgrid}
	Given a grid of hyperplanes there is a maximal one containing it. 
\end{lem}
\begin{proof}
	Suppose that there is a hyperplane $\hplt$ separating $\hpl_i, \hpl_{i+1}$ for some $i$. 
	Since $\hplv_j$ is transverse to $\hpl_i$ and $\hpl_{i+1}$ we see that $\hplv_j$ is transverse to $\hplt$. 
	Thus we can replace $H$ with the collection $\{\hpl_1, \dots, \hpl_i, \hplt, \hpl_{i+1},\dots,\hpl_q\}$. 
	Repeating this process we will eventually obtain a maximal grid:
	the process terminates since there are only finitely many hyperplanes separating any two hyperplanes. 
\end{proof}

\subsection{Hyperbolicity criteria}

Hyperbolicity in general CAT(0) spaces can be obtained by using the flat plane theorem \cite{bridson_existence_1995}. However, in cubical complexes there are also various other methods involving combinatorics of hyperplanes. 
We will make use of the following criterion from \cite{genevois2} which is proved in \cite{genevois}.

\begin{thm}[{\cite[Theorem 3.5]{genevois2},\cite[Theorem 3.3]{genevois}}]\label{thm:thingrids}
	Let $\tilde{X}$ be a finite dimensional CAT(0) cube complex. 
	Then $\tilde{X}$ is hyperbolic if and only if grids of hyperplanes are uniformly thin. 
\end{thm}

\begin{de}
	A {\em square} in a simplicial complex $\Gamma$ is any subgraph isomorphic to the boundary of a square (\emph{i.e.} a cycle of length $4$). A square in $\Gamma$ is \emph{flat} if it is a full subcomplex (\emph{i.e.} in $\Gamma$ there are no edges joining two diagonally opposite vertices of the square).
\end{de}

Flat squares in simplicial complexes are the only obstruction to hyperbolicity in right-angled Coxeter groups: 
\begin{theorem}[\cite{Mou88}]\label{thm:hyp RACGs}
	The right-angled Coxeter group $\CC\paren{\Gamma}$ is hyperbolic if and only if $\Gamma$ contains no flat squares. 
\end{theorem}

We wish to produce a similar criterion for non\=/positively curved CLCCs.

\

Let $\xcplx$ be a coupled link cube complex on $n$ colours and let $\tildexcplx$ denote its universal cover. If $\xcplx$ is not connected, $\tildexcplx$ is defined as the disjoint union of the universal covers of  the connected components of $\xcplx$.
The covering map induces a labelling on the vertices of $\tildexcplx$, where a vertex $w$ has label $(\fka,\fkb)\in\xcplx$ if it belongs to the preimage of $(\fka,\fkb)$. We will still say that $\fka$ and $\fkb$ are the coordinate simplices of $w$, and that the indices corresponding to vertices in $\fka$ (resp. $\fkb$) are the $A$\=/coordinates (resp. $B$\=/coordinates) of $w$. 
We still have that 
\[
 \lk(w,\tildexcplx)=\lk(\fka,\Gamma_A)\ast\lk(\fkb,\Gamma_B).
\]
Similarly, we say that an edge of $\xcplx$ is in the $k$\=/th coordinate if it connects two vertices whose labellings differ only on the $k$\=/coordinate.

Now let $\Gamma_A$ and $\Gamma_B$ be flag, so that $\tildexcplx$ is CAT(0). The same explaination of Subsection~\ref{ssec:hyperplanes in CLCC}, shows that with any hyperplane $\hpl\subsetneq \tildexcplx$ is naturally associated a colour $k\in\{1,\ldots,n\}$. Namely, $\hpl$ has colour $k$ if and only if it is an equivalence class of edges in the $k$\=/th coordinate. 
Of course, the hyperplane $\hpl$ cuts its carrier $N(\hpl)$ in half. Moreover, all the vertices on one half of $N(\hpl)$ will have the same $k$\=/th coordinate, say $a_k\in A_k$, and all the vertices on the other half will share a complementary $k$\=/th coordinate, say $b_k\in B_k$.

\begin{rmk}
 It is in fact possible to use this colouring of the hyperplanes to show that $\tildexcplx$ embeds into a product of $n$ trees (see also \cite{ChHa13}).
\end{rmk}

Recall, a simplicial complex is {\em 5-large} if it is flag and it contains no flat squares.%
\footnote{%
Equivalently, all geodesic closed cylces must have legth $5$ or more.
}
In the context of coupled link cube complexes the following weakening of 5\=/largeness will be sufficient to prove hyperbolicity. 

\begin{de}\label{def:pairwise}
	We say that a pair of $n$\=/coloured flag simplicial complexes $\Gamma_A, \Gamma_B$ are {\em pairwise 5-large} if whenever there is a flat square $(v_1, v_2, v_3, v_4)$ in one of the complexes, where $v_j$ has colour $k_j$, then there exists a $j \in\{ 1, \dots, 4\}$ such that the other complex contains no flat squares with two adjacent vertices of colours $k_j, k_{j+1}$. 
\end{de}

\begin{remark}\label{rem:5largeispairwise5large}
	If $\Gamma_A$ and $\Gamma_B$ are flag and one of them is 5-large, then the pair $\Gamma_A, \Gamma_B$ is pairwise 5-large. However, we will see many examples where the converse is false. 
\end{remark}

This condition allows us to extend one half of Theorem \ref{thm:hyp RACGs} to the context of coupled link cube complexes.

\begin{thm}\label{thm:hypclcc}
	If $\Gamma_A, \Gamma_B$ are pairwise 5-large, then $\pi_1(\xcplx, x_0)$ is hyperbolic for every base point $x_0\in\xcplx$.
\end{thm}
\begin{proof}
	Let $\tildexcplx$ be the universal cover of $\xcplx$. We need to show that every connected component of $\xcplx$ is hyperbolic. To do this, it is enough to show that in $\tildexcplx$ there are no maximal grids of hyperplanes of width $\geq 5$. Hyperbolicity then follows from Theorem~\ref{thm:thingrids} and Lemma~\ref{lem:maxgrid}. 
	
	Suppose by contradiction that we are given a maximal grid of hyperplanes of width $5$. 
	By maximality, given four hyperplanes $\hplv_i, \hplv_{i+1}, \hpl_j, \hpl_{j+1}$ there is no hyperplane separating $\hplv_i$ from $\hplv_{i+1}$. Therefore, the carriers $N(\hplv_i)$ and $N(\hplv_{i+1})$ must intersect. The same holds for $\hpl_j$ and $\hpl_{j+1}$. It then follows from Corollary \ref{cor:carriers} that the intersection of all four carriers contains some (possibly more than one) vertices of $\tildexcplx$.  
	We will choose one such vertex and call it $w_{i, j}$ see Figure \ref{fig:grid}.

\begin{figure}
    \centering
	\includegraphics{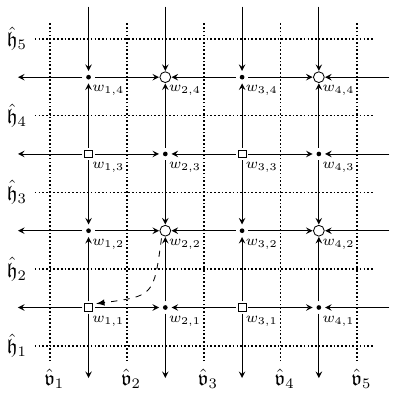}   
    \caption{A grid of hyperplanes of width $5$ and the associated points $w_{i,j}$. The arrows point to the side of the hyperplane having one fewer $A$\=/coordinate. 
    A vertex $w_{ij}$ with label $(\fka,\fkb)$ is marked with a white circle if the flat square is contained in $\lk(\fka,\Gamma_A)$. It is marked with a rectangles if the flat square is in $\lk(\fkb,\Gamma_B)$.}
    \label{fig:grid}
\end{figure}
	
	Each hyperplane $\hplv_i, \hplv_{i+1}, \hpl_j, \hpl_{j+1}$ uniquely determines a vertex in the link of $w_{i, j}$. We denote these vertices by $v_i, v_{i+1}, h_j,  h_{j+1}$ respectively. Since $\hplv_i$ and $\hpl_j$ are transverse, it follows from Lemma~\ref{lem:transverse hyperplane intersect} that $v_i$ and $h_j$ are joined by an edge. Similarly, we deduce that $(v_i, h_j, v_{i+1}, h_{j+1})$ is a square in the link of $w_{i,j}$.	
	Since $\hplv_i$ and $\hplv_{i+1}$ do not intersect, the corresponding vertices in the link are not linked by an edge. The same goes for $\hpl_j, \hpl_{j+1}$. We thus conclude that $(v_i, h_j, v_{i+1}, h_{j+1})$ is a flat square in the link of $w_{i,j}$. We shall denote this flat square $\square_{i,j}$. 
	
  Let $(\fka,\fkb)$ be the labelling $w$ and identify $\lk(w_{i, j},\tildexcplx)$ with the join $\lk(\fka,\Gamma_A)\ast\lk(\fkb,\Gamma_B)$. If three vertices of $\square_{i,j}$ were to belong to $\lk(\fka,\Gamma_A)$ while the fourth was in $\lk(\fkb,\Gamma_B)$, then $(v_i, h_j, v_{i+1}, h_{j+1})$ would not span a flat square. 
	Thus, there are only three possibilities: either all 4 vertices of $\square_{i,j}$ are in $\Gamma_A$, all 4 vertices $\square_{i,j}$ are in  $\Gamma_B$ or there are two vertices of $\square_{i,j}$ in $\Gamma_A$ and two vertices in $\Gamma_B$. Moreover, in the last case the two vertices in $\Gamma_A$ cannot be adjacent in $\square_{i,j}$.

Let $k^v_{i}$ and $k^h_j$ be the colours of the hyperplanes $\hplv_i$ and $\hpl_j$ respectively. If the $k^v_{i+1}$\=/th coordinate of $w_{i,j}$ is an $A$\=/coordinate, then it will necessarily be a $B$\=/coordinate for $w_{i+1,j}$. Moreover, $w_{i,j}$ and $w_{i+1,j}$ must have the same $k^h_j$\=/th and $k^h_{j+1}$\=/th coordinate because they lie on the same side of the carriers of $\hpl_j$ and $\hpl_{j+1}$. It follows that $\square_{i+1,j}$ will have exactly two more vertices in $\Gamma_A$ than $\square_{i,j}$. As a consequence, we see that exactly one among $\square_{2,2},\square_{2,3},\square_{3,2}$ and $\square_{3,3}$ is fully contained in $\Gamma_A$ (and exactly one of them is fully contained in $\Gamma_B$).

For concreteness, assume that $\square_{2,2}$ is completely contained in $\Gamma_A$ (the other three cases are similar). The colours of $\square_{i,j}$ are $k^v_{2},k^h_2,k^v_{3},k^h_3$. By pairwise $5$\=/largeness, there should be two adjacent colours such that there are no flat squares in $\Gamma_B$ having those same colours on adjacent vertices. 
Again, assume for concreteness that these two colours are $k^v_{2},k^h_2$ (the other cases being analogous). 
It follows from the previous argument that $\square_{2,1}$ has two vertices in $\Gamma_B$ and two vertices in $\Gamma_A$, and that $\square_{1,1}$ is fully contained in $\Gamma_B$. Moreover, $\square_{i,j}$ has $k^v_{2},k^h_2$ as colours of adjacent vertices, contradicting pairwise $5$\=/largeness (Figure~\ref{fig:grid}).
\end{proof}

The criterion provided by Theorem~\ref{thm:hypclcc} is rather flexible, but it is not sharp (see \emph{e.g.} Example~\ref{eg:hyperbolic non 5large}). In fact, it appears to be very complicated to provide an ``if and only if'' characterisation of hyperbolicity for coupled link cube complexes. Rather, it seems to be more cost\=/effective to prove various partial criteria. We shall now give one more such criterion.

\begin{de}
 We say that a square in an $n$\=/coloured simplicial complex $\Gamma$ is \emph{bicoloured} if it has vertices of only two colours $k_1$, $k_2$ (necessarily alternating). We say that $\Gamma$ has \emph{only bicoloured flat squares} if every flat square is bicoloured.
 
 A pair of bicoloured squares $c_A$, $c_B$ in two $n$\=/coloured simplicial complexes $\Gamma_A$, $\Gamma_B$ are \emph{complementary} if they both have colours $k_1$, $k_2$ and there exist simplices $\fka\subsetneq \Gamma_A$ and $\fkb\subsetneq \Gamma_B$ such that $c_A\subseteq\lk(\fka,\Gamma_A)$, $c_B\subseteq(\fkb,\Gamma_B)$ and the simplex $\fka^{k_1,k_2}$ is complementary to $\fkb$.
\end{de}

In other words, two squares are complementary if they belong to the link of two (possibly empty) simplices $\fka$, $\fkb$ that have no vertex of coordinate $k_1$, $k_2$, and so that for every other $k\in\{1,\ldots,n\}$ exactly one of $\fka$ and $\fkb$ has a vertex on the $k$\=/th coordinate.

The CLCC associated with the pair $\fka\ast c_A$ and $c_B\ast\fkb$ is equal to the double  cubical barycentric subdivision of the torus $\TT^2$. Moreover, if $c_A$ and $c_B$ are flat squares then $\fka\ast c_A\subseteq \Gamma_A$ and $c_B\ast\fkb\subseteq\Gamma_B$ are full subcomplexes. It follows from functoriality (Lemma~\ref{lem:local isometry}) that if $\Gamma_A$, $\Gamma_B$ contain complementary bicoloured flat squares the complex $\xcplx$ contains a locally isometrically embedded copy of the torus. If $\Gamma_A$ and $\Gamma_B$ are flag, it follows that the universal cover $\tildexcplx$ has a flat plane as a subcomplex. In particular, it is not hyperbolic. The next result shows that, under the assumption that there are only bicolour flat squares, the converse is also true.

\begin{thm}\label{thm:bicolour criterion}
 Let $\Gamma_A$ and $\Gamma_B$ be $n$\=/coloured flag simplicial complexes with only bicoloured flat squares. 
If $\pi_1(\xcplx,x_0)$ is not hyperbolic, then there exists a pair of complementary flat squares. In particular, the connected component of $x_0$ in $\xcplx$ must contain as a subcomplex a locally isometric embedding of a torus.
\end{thm}
\begin{proof}
 Assume that $\pi_1(\xcplx,x_0)$ is not hyperbolic, so that the universal cover of that connected component is not hyperbolic. As in the proof of Theorem~\ref{thm:hypclcc}, choose a maximal grid of hyperplanes of width $3$ and vertices $w_{i,j}$ belonging to the intersections of the carriers of $\hplv_i, \hplv_{i+1}, \hpl_j, \hpl_{j+1}$. These hyperplanes determine flat squares $\square_{i,j}$ in the links of $w_{i,j}$. It follows from the only bicoloured condition that all the vertical hyperplanes and all the horizontal hyperplanes have the same colours, say $k^v$ and $k^h$. 
 
 As before, exactly one of $\square_{1,1},\ \square_{1,2},\ \square_{2,1},\ \square_{2,2}$ is fully contained in $\Gamma_B$. Without loss of generality, we can assume it is $\square_{1,1}$. 
 To prove the theorem it will be enough to show that we can choose $w_{2,2}$ to be the vertex obtained from $w_{1,1}$ by crossing only $\hplv_2$ and $\hpl_2$. In fact, if this is the case we let $c_B$ and $c_A$ be the bicoloured flat squares determined by $w_{1,1}$ and $w_{2,2}$. We then see that they are complementary by letting 
 $\fkb$ (resp. $\fka$) be the simplex of $B$\=/coordinates (resp. $A$\=/coordinates) of $w_{1,1}$  (resp. $w_{2,2}$).
 
 Let $u_{1,2}$ be the vertex obtained from $w_{1,1}$ by crossing $\hpl_2$. Since $u_{1,2}$ belongs to $Y\coloneqq N(\hpl_2)\cap N(\hplv_1)\cap N(\hplv_2)$, it is enough to show that it also belongs to $N(\hpl_3)$ to conclude that we can let $w_{1,2}=u_{1,2}$. Since $Y$ is connected, we can choose a shortest path $\gamma=(\gamma_0,\dots,\gamma_l)$ in $Y^{(1)}$ connecting $u_{1,2}$ to $Y\cap N(\hpl_3)$. We aim to show that $l=0$, \emph{i.e.} $\gamma$ is in fact the constant path. 
 Note that all the vertices in $\gamma$ must have the same $k^v$\=/th and $k^h$\=/th coordinate. 
 
 Assume toward a contradiction, that $l\geq 1$, and let $\hpl'$ be the hyperplane separating $\gamma_{l-1}$ and $\gamma_{l}$. Since the path is contained $N(\hplv_1)\cap N(\hplv_2)$, the hyperplane $\hpl'$ is transverse to $\hplv_1$ and $\hplv_2$. It follows that $\hplv_1, \hpl_3,\hplv_3$ and $\hpl'$ define a square in the link of $\gamma_l$. Since $\gamma_{l-1}$ and $\gamma_{l}$ have the same $k^h$\=/th coordinate,  the hyperplane $\hpl'$ must be of a colour different from $k^h$. We then see that the square in not bicoloured, and hence $\hpl'$ and $\hpl_3$ must be transverse. It follows that $\gamma_{l-1}\in N(\hpl_3)$, contradicting the assumption (Figure~\ref{fig:walking_across_cubes}).

\begin{figure}
    \centering
	\includegraphics{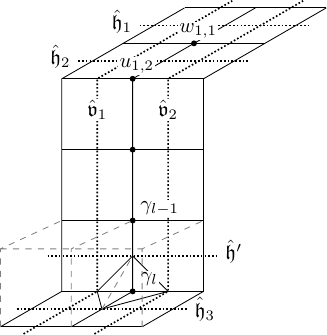}   
    \caption{Since the square around $\gamma_l$ cannot be flat, the dashed cubes must belong to $\tildexcplx$. It follows that $\gamma_{l-1}$ belongs to the carrier of $\hpl_3$ as well.}
    \label{fig:walking_across_cubes}
\end{figure}
 
 The same argument shows that we can let $w_{2,2}$ be the vertex $u_{2,2}$ obtained from $u_{1,2}$ by crossing $\hplv_2$, thus proving the theorem.
\end{proof}

\subsection{Hyperbolic examples}

We can now detail many examples of coupled link cube complexes with hyperbolic fundamental group. An immediate corollary to Theorem \ref{thm:hypclcc} and Remark \ref{rem:5largeispairwise5large} is the following:

\begin{cor}\label{cor:one.5-large.hyperbolic.CLCC}
	If $\Gamma_A$ is $5$\=/large, then each connected component of $\tildexcplx$ is hyperbolic.
\end{cor}

\begin{rmk}
	Since we realised the (commutator subgroup) of a right angled Coxeter group $\CC(\Gamma)$ as the fundamental group of $\BX_{\Gamma,\SS^{n-1}}$, this corollary gives us yet another proof that $\CC(\Gamma)$ is hyperbolic if and only if $\Gamma$ is $5$\=/large.
\end{rmk}

From \cite[Proposition 2.13]{przytycki_flag-no-square_2009} it follows that there exist $5$\=/large triangulations of the $3$\=/sphere. Specifically, \cite[Lemma 2.1]{przytycki_flag-no-square_2009} shows that the boundary of the 600\=/cell is such an example. We obtain the following: 

\begin{thm}\label{thm:4manifolds}
	Let $\Gamma_B$ be a $5$\=/large PL triangulation of $\SS^3$ with an arbitrary $n$\=/colouring having at least one vertex for each colour and let $\Gamma_A=\Ast_{i=1}^n \SS^0$. Then $\xcplx$ is a $4$\=/dimensional closed connected PL manifold with hyperbolic fundamental group.
\end{thm}
\begin{proof}
	We have already shown in Proposition \ref{prop:manifoldsifspheres} that $\xcplx$ is a closed PL manifold. Since $\Gamma_A$ is the complete $n$\=/coloured complex $\SS^0\ast\cdots\ast\SS^0$ (which is strongly connected) and $\Gamma_B$ has at least one vertex per colour, it follows easily from Lemma~\ref{lem:connect.criterion} that $\xcplx$ is connected.\footnote{%
	If there was some colour such that $\Gamma_B$ has no vertices of that colour, $\xcplx$ would be a disjoint union of isomorphic PL manifolds. This can be proved using Proposition~\ref{prop:components are CLCC} to show that all the connected components are CLCC defined by the same simplicial complexes.
	}
	Hyperbolicity of the fundamental group follows from Corollary~\ref{cor:one.5-large.hyperbolic.CLCC}.
\end{proof}

We can in fact prove that in low dimensions there are no topological obstructions to a pair of simplicial complexes having a pairwise 5-large colouring. We will prove it by considering barycentric subdivisions.

\begin{de}
	Let $\Gamma$ be an $n$-coloured simplicial complex. Let $C_1, \dots, C_n$ be the colours of this complex. Let $\square = v_1, \dots, v_4$ be a square in $\Gamma$. If $v_1\in C_i, v_2\in C_j, v_3\in C_k$ and $v_4\in C_l$, we say that $\square$ is of {\em type $[i, j, k, l]$}. The type of a square is defined up to cyclic permutations and reflections.
\end{de}

\begin{lem}\label{lem:subbicolour}
	Let $\Gamma$ be a $2$\=/dimensional simplicial complex and let $\Gamma'$ be its barycentric subdivision. Let $\Gamma'$ be $3$\=/coloured with colours $V, E, F$ where $V (E, F$, respectively) corresponds to vertices of $\Gamma'$ that are the barycentres of vertices (edges, faces respectively). 
	Then any flat square is of type $[V,F,V,F]$. 
\end{lem}
\noindent\textit{Proof. }
 Besides squares of type $[V,F,V,F]$, there are five more possible cycles of length $4$. We examine them in turn to show that they are not flat squares (in what follows $v_i$ denotes a vertex in $V$, $e_i$ in $E$ and $f_i$ in $F$):
	\begin{itemize}
		\item The square has vertices $v_1, e_1, v_2, e_2$. Since $e_1$ is connected to both $v_1$ and $v_2$, it must be the barycenter of the edge between these two vertices. The same holds for $e_2$ and so $e_1 = e_2$. 
		\item The square has vertices $f_1, e_1, f_2, e_2$. In this case $e_1$ and $e_2$ correspond to edges of the triangles $f_1$ and $f_2$. Since any triangle is determined by two of its edges, we see that $f_1 = f_2$. 
		\item The square has vertices $v_1,e,v_2,f$. Then $v_1$ and $v_2$ belong to both $f$ and $e$. Since $\Gamma$ is simplicial, $e$ is uniquely determined by $v_1$ and $v_2$ and therefore it is contained in $f$. It follows that $(e,f)$ is an edge in $\Gamma'$.
		\item The square has vertices $e_1,v,e_2,f$. Then $e_1$ and $e_2$ contain $v$ and are contained in $f$. Since $\Gamma$ is simplicial, both $v$ and $f$ are uniquely determined by $e_1$ and $e_2$ and therefore $v$ is contained in $f$.
		\item The square has vertices $f_1,v,f_2,e$. Then $f_1$ and $f_2$ contain $v$ and $e$. Since $\Gamma$ is simplicial, the intersection of $f_1$ and $f_2$ is $e$ and therefore $v$ is contained in $e$. \qed
	\end{itemize}

\begin{thm}\label{thm:barycentric.subdivision.hyperbolic}
	Let $\Gamma$ and $\Lambda$ be $2$\=/dimensional simplicial complexes. Let $\Gamma'$ be $3$\=/coloured as in Lemma~\ref{lem:subbicolour} and let $\Lambda'$ be $3$\=/coloured by taking a permutation of the colours given in Lemma~\ref{lem:subbicolour} that does not fix $E$. Then each connected component of $\widetilde{\BX}_{\Gamma',\Lambda'}$ is hyperbolic.
	
	Moreover, if $H_2(\Gamma,\ZZ_2)$ and $H_2(\Lambda,\ZZ_2)$ are not trivial, then $H_3(\xcplx,\ZZ_2)$ is also non\=/trivial.
\end{thm}
\begin{proof}
	By Lemma \ref{lem:subbicolour}, in $\Gamma'$ and $\Lambda'$ there is only one type of flat square: $[V,F,V,F]$ for the former and $[E,X,E,X]$ for the latter, where $X$ is either $V$ or $F$ depending on the permutation chosen for the colours of $\Lambda'$. In either case, these complexes are pairwise 5-large. Thus we obtain hyperbolicity from Theorem \ref{thm:hypclcc}. 
	
	The statement about homology follows immediately from Theorem~\ref{thm:chains.to.chains,cycles.to.cycles,amen} because $3$\=/coloured $2$\=/cycles are necessarily smartly paired and therefore define a $3$\=/cycle in $\xcplx$. This $3$\=/cycle cannot be a boundary because $\xcplx$ is a $3$\=/dimensional complex.
\end{proof}

We can proceed in a similar way for 3-dimensional simplicial complexes. Here there are more flat squares. 

\begin{lem}\label{lem:squaresin3dim}
	Let $\Gamma$ be a $3$\=/dimensional simplicial complex and let $\Gamma'$ be equipped with the $4$\=/colouring by colours $\{0,1,2,3\}$, where vertices of colour $i$ corresponds the barycentres of $i$-cells. 
	Then any flat square has one of the following types: 
	\begin{itemize}
		\item $[0, 2, 0, 2]$,
		\item $[0, 2, 0, 3]$,
		\item $[0, 3, 0, 3]$, 
		\item $[0, 3, 1, 3]$, 
		\item $[1, 3, 1, 3]$.
	\end{itemize}
\end{lem}
\begin{proof}
 It follows from Lemma \ref{lem:subbicolour} that any flat square contained in the $2$\=/skeleton of $\Gamma'$ must be of type $[0,2,0,2]$. It only remains to study squares $(v_1,v_2,v_3,v_4)$ of type $[x,y,z,3]$ with $x,y,z\in\{0,1,2,3\}$. If $y<z$, the three vertices $v_2,v_3,v_4$ correspond to a flag of simplices of increasing dimension in $\Gamma$. Hence the square cannot be flat, as there is an edge between $v_2$ and $v_4$. The same argument shows that $y$ cannot be smaller than $x$. We can thus assume $y>\max\{x,z\}$. Taking a reflection if necessary, we can further assume that $x\leq z$.
 
 For $(v_1,v_2,v_3,v_4)$ to be flat, it is necessary that the simplex of $\Gamma$ associated with $v_2$ is not contained in the simplex of $\Gamma$ associated with $v_4$. As a consequence, the intersection of the simplices of $\Gamma$ associated with $v_2$ and $v_4$ is a simplex of dimension at most $y-1$. In order for $v_1$ and $v_3$ to be connected to both $v_2$ and $v_4$, it is necessary that the simplices of $\Gamma$ associated with $v_1$ and $v_3$ are contained in the intersection of the simplices of $\Gamma$ associated with $v_2$ and $v_4$. However, we now see that if $z=y-1$  then the simplex of $\Gamma$ associated with $v_3$ contains the simplex of $\Gamma$ associated with $v_1$ and hence the square is not flat. We deduce that both $x$ and $z$ must be at most $y-2$.
 
 The squares satisfying these conditions are precisely those listed in the statement of the lemma. It is not hard to verify that they can all give rise to flat squares.
\end{proof}

Now that we understand all the possible flat squares in barycentric subdivisions of $3$\=/dimensional complexes we can give more examples of hyperbolic coupled link cube complexes. 

\begin{thm}\label{thm:3dimbarycentricsubdivision}
	Let $\Gamma$ and $\Lambda$ be $3$\=/dimensional simplicial complexes. Then there is a choice of $4$\=/colouring that makes $\Gamma', \Lambda'$ pairwise 5-large (and hence each component of $\tildexcplx$ is hyperbolic). 
	
	Moreover, if $H_3(\Gamma,\ZZ_2)$ and $H_3(\Lambda,\ZZ_2)$ are not trivial, then $H_4(\xcplx,\ZZ_2)$ is also non\=/trivial.
\end{thm}
\begin{proof}
	Let $\Gamma'$ be $4$\=/coloured by $\{0,1,2,3\}$ as in Lemma \ref{lem:squaresin3dim}. Let $\Lambda$ be $4$\=/coloured by permuting the colours from Lemma~\ref{lem:squaresin3dim} as follows:
	\[
	\begin{tikzcd}[row sep=2 ex]
	 \text{vertices} \arrow[d] &
	 \text{edges} \arrow[d] &
	 \text{faces} \arrow[d] &
	 \text{tetrahedra} \arrow[d] \\
	 1 & 3 & 0  & 2.
	\end{tikzcd}
    \]	
	
	It follows from Lemma \ref{lem:squaresin3dim} that $\Gamma'$ has flat squares of types $[0, 2, 0, 2]$, $[0, 2, 0, 3]$, $[0, 3, 0, 3]$, $[0, 3, 1, 3]$ and $[1, 3, 1, 3]$. 
	
	It also follows that $\Lambda'$ has flat squares of type $[1, 0, 1, 0]$, $[1, 0, 1, 2]$, $[1, 2, 1, 2]$, $[1, 2, 3, 2]$ and $[3, 2, 3, 2]$.
	
	Thus we can see that with this colouring this pair of complexes is pairwise 5-large. Thus we obtain hyperbolicity from Theorem \ref{thm:hypclcc}. 
	
	The statement about homology follows immediately from Theorem~\ref{thm:chains.to.chains,cycles.to.cycles,amen} because $4$\=/coloured $3$\=/cycles are necessarily smartly paired and therefore define a $4$\=/cycle in $\xcplx$. This $4$\=/cycle cannot be a boundary because $\xcplx$ is a $4$\=/dimensional complex.
\end{proof}

We have so far seen no topological obstructions to a pair of complexes admitting a colouring that makes them pairwise 5-large. It would be of interest to know if there is an example of a pair of topological spaces such that every triangulation and every colouring is not pairwise 5-large. In particular we raise the following question:
\begin{qu}\label{qu:hypmanifold}
	Is there a bound on the dimension of coupled link cube complexes that are PL manifolds with hyperbolic fundamental group?
\end{qu}

In the case of hyperbolic right-angled Coxeter groups, the question above has an affirmative answer \cite[Theorem 2]{januszkiewicz_hyperbolic_2003}. Namely, if the Davis complex is a PL manifold and the right-angled Coxeter group is hyperbolic, then it has dimension $\leq 4.$
In fact, \cite{januszkiewicz_hyperbolic_2003} shows that if a RACG is a virtual Poincaré duality group then it has dimension $\leq 4$. 
Thus, we also ask the more general question: 
\begin{qu}\label{qu:poincareduality}
	Is there a bound on the dimension for which $\pi_1(\xcplx)$ can be a hyperbolic Poincaré duality group? 
\end{qu}

\begin{rmk}
 Taking $n$\=/coloured barycentric subdivision of complexes of dimension $n$ cannot yield hyperbolic CLCCs for any $n\geq 5$. One can check that the number of bicoloured flat squares increases too quickly and hence any such colouring will give rise to complementary bicoloured flat squares.
\end{rmk}

We conclude our list of examples by giving an instance where Theorem~\ref{thm:bicolour criterion} applies while Theorem~\ref{thm:hypclcc} does not.

\begin{exmp}\label{eg:hyperbolic non 5large}
  Let both $\Gamma_A$ and $\Gamma_B$ be equal to the $4$\=/coloured torus as depicted in Figure~\ref{fig:hexagonal_torus}. Then $\xcplx$ has hyperbolic fundamental group by Theorem~\ref{thm:bicolour criterion} (one can also check that it is a cubulated surface with $640$ vertices and genus $63$). However, $\Gamma_A$ and $\Gamma_B$ are not pairwise $5$\=/large.
\end{exmp}

\begin{figure}
    \centering
	\includegraphics{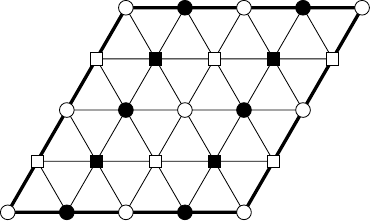}   
    \caption{A $4$\=/coloured triangulation of the torus (the side edges and the horizontal edges are pairwise identified). Every combination of two colours appears as a bicoloured flat square.}
    \label{fig:hexagonal_torus}
\end{figure}

\bibliographystyle{plain}
\bibliography{bibshort,morebib}

\begin{thebibliography}{10}

\bibitem{bestvina_morse_1997}
M.~Bestvina and N.~Brady.
\newblock Morse theory and finiteness properties of groups.
\newblock {\em Invent. Math.}, 129(3):445--470, 1997.

\bibitem{bieri_homological_1981}
Robert Bieri.
\newblock {\em Homological dimension of discrete groups}.
\newblock Queen {Mary} {College} {Mathematical} {Notes}. Queen Mary College,
  Department of Pure Mathematics, London, second edition, 1981.

\bibitem{bridson_existence_1995}
M.~R. Bridson.
\newblock On the existence of flat planes in spaces of nonpositive curvature.
\newblock {\em Proceedings of the American Mathematical Society},
  123(1):223--223, January 1995.

\bibitem{bridson_metric_1999}
M.~R. Bridson and A.~Haefliger.
\newblock {\em Metric spaces of non-positive curvature}, volume 319 of {\em
  Grundlehren der {Mathematischen} {Wissenschaften} [{Fundamental} {Principles}
  of {Mathematical} {Sciences}]}.
\newblock Springer-Verlag, Berlin, 1999.

\bibitem{brown_cohomology_1982}
K.~S. Brown.
\newblock {\em Cohomology of {Groups}}, volume~87 of {\em Graduate {Texts} in
  {Mathematics}}.
\newblock Springer New York, New York, NY, 1982.

\bibitem{ChHa13}
Victor Chepoi and Mark~F Hagen.
\newblock On embeddings of {CAT(0)} cube complexes into products of trees via
  colouring their hyperplanes.
\newblock {\em Journal of Combinatorial Theory, Series B}, 103(4):428--467,
  2013.

\bibitem{genevievehaulmark}
Pallavi Dani, Matthew Haulmark, and Genevieve Walsh.
\newblock Right-angled coxeter groups with non-planar boundary, 2019.

\bibitem{davis1983groups}
Michael~W Davis.
\newblock Groups generated by reflections and aspherical manifolds not covered
  by euclidean space.
\newblock {\em Annals of Mathematics}, pages 293--324, 1983.

\bibitem{Dro03}
Carl Droms.
\newblock A complex for right-angled coxeter groups.
\newblock {\em Proceedings of the American Mathematical Society},
  (8):2305--2311, 2003.

\bibitem{genevois}
Anthony Genevois.
\newblock Coning-off cat(0) cube complexes, 2016.

\bibitem{genevois2}
Anthony Genevois.
\newblock Contracting isometries of {${\rm CAT}(0)$} cube complexes and
  acylindrical hyperbolicity of diagram groups.
\newblock {\em Algebr. Geom. Topol.}, 20(1):49--134, 2020.

\bibitem{haglund_special_2008}
F.~Haglund and D.~T. Wise.
\newblock Special cube complexes.
\newblock {\em Geom. Funct. Anal.}, 17(5):1551--1620, 2008.

\bibitem{hatcher_algebraic_2002}
Allen Hatcher.
\newblock {\em Algebraic topology}.
\newblock Cambridge University Press, Cambridge, 2002.

\bibitem{januszkiewicz_hyperbolic_2003}
Tadeusz Januszkiewicz and Jacek Świątkowski.
\newblock Hyperbolic {Coxeter} groups of large dimension.
\newblock {\em Comment. Math. Helv.}, 78(3):555--583, July 2003.

\bibitem{kropholler_hyperbolic}
Robert Kropholler.
\newblock Hyperbolic groups with almost finitely presented subgroups.
\newblock {\em Groups, Geometry, and Dynamics (to appear)}.

\bibitem{Mou88}
Gabor Moussong.
\newblock {\em Hyperbolic coxeter groups}.
\newblock PhD thesis, The Ohio State University, 1988.

\bibitem{osajda_construction_2013}
Damian Osajda.
\newblock A construction of hyperbolic {Coxeter} groups.
\newblock {\em Comment. Math. Helv.}, 88(2):353--367, April 2013.

\bibitem{przytycki_flag-no-square_2009}
Piotr Przytycki and Jacek Świątkowski.
\newblock Flag-no-square triangulations and {Gromov} boundaries in dimension 3.
\newblock {\em Groups Geom. Dyn.}, 3(3):453--468, 2009.

\bibitem{roller}
Martin Roller.
\newblock Poc sets, median algebras and group actions, 2016.

\bibitem{sageev_cat0_2014}
Michah Sageev.
\newblock {CAT}(0) cube complexes and groups.
\newblock In {\em Geometric group theory}, volume~21 of {\em {IAS}/{Park}
  {City} {Math}. {Ser}.}, pages 7--54. Amer. Math. Soc., Providence, RI, 2014.

\bibitem{zeeman1966seminar}
Erik~Christopher Zeeman.
\newblock {\em Seminar on combinatorial topology}.
\newblock Institut des hautes etudes scientifiques, 1966.

\end{thebibliography}

\end{document}